\documentclass[a4paper]{amsart}
\usepackage{verbatim}
\usepackage[dvips]{color}
\usepackage{amssymb}

\author[F. Duzaar]{Frank Duzaar}
\address{Frank Duzaar\\Department Mathematik, Universit\"at
Erlangen--N\"urnberg\\ Bismarckstrasse 1 1/2, 91054 Erlangen, Germany}
\email{duzaar@mi.uni-erlangen.de}

\author[G. Mingione]{Giuseppe Mingione}
\address{Dipartimento di Matematica, Universit\`a di Parma\\
Viale G.~P.~Usberti 53/a, Campus, 43100 Parma, Italy}
\email{giuseppe.mingione@unipr.it.}

\allowdisplaybreaks

\newtheorem{theorem}{Theorem}[section]

\newtheorem{lemma}{Lemma}[section]

\theoremstyle{definition}

\newtheorem{remark}{Remark}[section]

\numberwithin{equation}{section}

\newcommand{\rif}[1]{(\ref{#1})}

\newcommand\ap{``}

\def\eqn#1$$#2$${\begin{equation}\label#1#2\end{equation}}
\def\charfn_#1{{\raise1.2pt\hbox{$\chi
_{\kern-1pt\lower3pt\hbox{{$\scriptstyle#1$}}}$}}}

\def\qq1{q_*}
\def\q2{q_{**}}

\def\ep{\varepsilon}
\def\en{\mathbb N}
\def\er{\mathbb R}

\def\loc{\operatorname{loc}}

\def\ma{\mathbb R^{Nn}}

\newdimen\vintbar
\def\vint{-\kern-\vintbar\int}

\def\A{\mathcal A}
\def\B{\mathcal B}

\def\tr{\tilde R}

\def\0{\boldsymbol 0}

\newcommand{\ratio}{\nu, L}

\newcommand{\divo}{\textnormal{div}}

 \newcommand{\mean}[1]{-\hskip-1.08em\int_{#1}}

\newcommand{\npma}{{\bf W}_{\frac{1}{p},p}^{\mu}}

\newcommand{\trif}[1] {\textnormal{\rif{#1}}}

\newtoks\by
\newtoks\paper
\newtoks\book
\newtoks\jour
\newtoks\yr
\newtoks\pages
\newtoks\vol
\newtoks\publ
\def\et{ \& }
\def\name[#1, #2]{#1 #2}
\def\ota{{\hbox{\bf ???}}}
\def\cLear{\by=\ota\paper=\ota\book=\ota\jour=\ota\yr=\ota
\pages=\ota\vol=\ota\publ=\ota}
\def\endpaper{\the\by, \textit{\the\paper},
{\the\jour} \textbf{\the\vol} (\the\yr), \the\pages.\cLear}
\def\endbook{\the\by, \textit{\the\book},
\the\publ, \the\yr.\cLear}
\def\endpap{\the\by, \textit{\the\paper}, \the\jour.\cLear}
\def\endproc{\the\by, \textit{\the\paper}, \the\book, \the\publ,
\the\yr, \the\pages.\cLear}

\title[Local Lipschitz regularity for degenerate elliptic systems] {Local Lipschitz regularity for degenerate elliptic systems}
\begin{document}\maketitle

\begin{abstract}
We start presenting an $L^{\infty}$-gradient bound for solutions to non-homogeneous $p$-Laplacean type systems and equations, via suitable non-linear potentials of the right hand side. Such a bound implies a Lorentz space characterization of Lipschitz regularity of solutions which surprisingly turns out to be independent of $p$, and that reveals to be the same classical one for the standard Laplacean operator. In turn, the a priori estimates derived imply the existence of locally Lipschitz regular solutions to certain degenerate systems with critical growth of the type arising when considering geometric analysis problems, as recently emphasized by Rivi\`ere \cite{rivinv}.
\end{abstract}
\section{Introduction and results} In this paper we consider a class of elliptic systems and scalar equations with $p$-growth of Schr\"odinger type
\eqn{aapp2}
$$
-\divo\, a(Du)=b(x,u,Du)V(x)\,, \quad  \qquad \mbox{in}\ \Omega$$ whose model involves the degenerate $p$-Laplacean operator:
\eqn{ppp0}
$$
-\divo\, (|Du|^{p-2}Du)=|Du|^qV(x) \qquad 0 \leq q \leq p-1\,.
$$
The above systems are considered in a bounded, Lipschitz domain $\Omega \subset \er^n$, with $n \geq 2$ while in general we shall assume at least that $V \in L^2(\Omega, \er^N)$; the higher dimensional case $n\geq 3$ will be of special importance for the connections we shall outline. The solutions considered are the usual energy ones coming for the spaces naturally associated to the considered operator, i.e. distributional solutions $u\colon\Omega \to \er^N$ with $u \in W^{1,p}(\Omega,\er^N)$, $N \geq 1$; as usual, for the results we are going to consider,  we assume also $p>1$. The main emphasis will the be on the extreme cases of the system \rif{ppp0} with respect to $q$, that is $q=0$
\eqn{typep}
$$
-\triangle_ p u \equiv -\divo\, (|Du|^{p-2}Du)=V
$$
and $q=p-1$
\eqn{critical}
$$
-\divo\, (|Du|^{p-2}Du)=|Du|^{p-1}V(x)\,.
$$
Further emphasis is then given on the conformal case $p=n>2$
\eqn{conf}
$$
-\divo\, (|Du|^{n-2}Du)=|Du|^{n-1}V(x)\,.
$$
In fact, the problems treated are related to a few recent and fundamental results of Rivi\`ere \cite{rivinv}, who in the two-dimensional case $n=2$ considered systems of the type
\eqn{riviere}
$$
-\triangle u =  W(x)\cdot Du\,,
$$
where $W(x)$ is a antisymmetric tensor field; in this case a main result in \cite{rivinv} is that $Du \in L^{2+\delta}$ for some $\delta>0$, which in turn implies the H\"older continuity of $u$. As explained in \cite{rivnotes} an additional bootstrap argument eventually leads to $Du \in L^q$ for every $q < \infty$.
The general $n>2$ dimensional version of the system \rif{riviere} which is relevant for the problems presented in \cite{rivinv, rivnotes} is given by the conformal system
\eqn{rivn}
$$
-\triangle_n u =  |Du|^{n-2} W(x)\cdot Du\,,
$$
which clearly falls in the realm of \rif{conf}, and for which one would like to prove $Du \in L^{n+\delta}$ according to Rivi\`ere's far reaching theory. At this stage, as emphasized by Rivi\`ere, another, stronger meaningful condition to consider on the vector field $W$ is that it belongs to the limiting Lorentz space space $L(n,1)$; therefore, the subtler anti-symmetry condition - which in this setting works as a {\em cancellation} condition - is replaced by a limiting {\em size} condition. The identification of the limiting $L(n,1)$ condition was first done by H\'elein in his beautiful work on the moving frames techniques, introduced to treat cases of maps taking values in manifold without symmetries; see \cite{He1, He2}.  With such an assumption the system in \rif{rivn} now falls in the realm of those in \rif{conf}, and in this case we can prove an existence and regularity theorem which in particular yields locally Lipschitz solutions. What of course differs here from what is considered in \cite{rivinv, rivnotes} is the a priori regularity for any local solution of \rif{conf} required there, since we are proving here regularity only for solutions obtained by approximation via suitable regularized problems. On the other hand an interesting point is that the a priori estimates are independent - as it must happen in such cases - of the particular approximation chosen; see also Remark \ref{indepe}.

More in general, the results we are presenting in this paper address the issue of the boundedness of the gradient of the non-homogeneous $p$-Laplacean system and give sharp answers to the question. In particular, we show that when the right hand side function $V$ belongs to the limiting Lorentz space $L(n,1)$ the gradient of the solution considered is locally bounded.

\subsection{Results} We consider equations/systems of the type
\eqn{ppp}
$$
-\divo\, a(Du)=V(x)\,.
$$
The continuous vector field $a \colon \Omega \times \ma \to \ma$ is assumed to be $C^1$-regular in the gradient variable $z$ in $\ma$ when $p\geq 2$ and in $\er^n\setminus \{0\}$ if $1<p<2$ and $s>0$. The following standard {\em growth and ellipticity assumptions} are assumed:
\eqn{asp}
$$
\left\{
    \begin{array}{c}
    |a(z)|+|\partial a(z)|(|z|^2+s^2)^{\frac{1}{2}} \leq L(|z|^2+s^2)^{\frac{p-1}{2}} \\ [3 pt]
    \nu^{-1}(|z|^2+s^2)^{\frac{p-2}{2}}|\lambda|^{2} \leq \langle \partial a(z)\lambda, \lambda
    \rangle
    \end{array}
    \right.
$$
whenever $z \in \ma \setminus\{0\}$ and $\lambda \in \ma$, where
$0< \nu\leq L$ and $s\geq 0$ is a fixed parameter introduced to distinguish the degenerate case ($s=0$) from the non-degenerate one ($s>0$).
The prototype of \rif{ppp} is - choosing $s=0$ - clearly given by the $p$-Laplacean equation/system in \rif{typep}. On the right hand side vector field $V$ we shall initially assume that is an $L^2$-map with values in $\er^N$, defined on $\Omega$; eventually letting $V\equiv 0$ in $\er^n \setminus \Omega$, without loss of generality we may assume that $V \in L^2(\er^n, \er^N)$. Moreover, although a lighter assumption can be considered, in order to consider only the core aspects of the problem - that is the a priori estimates - we shall give a treatment in the context of the usual monotonicity methods, therefore not considering so called very weak solutions - that is distributional solutions not lying in $W^{1,p}$. We shall therefore always assume that the right hand side $V$ belongs to the dual space $W^{-1,p'}$. In turn, for this it will be sufficient to have that $V \in L^{(p^*)'}$ therefore the minimal requirement on $V$ will be
\eqn{VVV}
$$
V \in L^{m}(\er^n, \er^N)\,,\qquad   m:=\max\, \left\{2,\frac{np}{np-n+p}\right\}\,.
$$
Results of the type considered here when $V \in L^{2}\setminus W^{-1,p'}$ - which on the other hand only happens when $p< 2n/(n+2)$ - can be still obtained, but these involve as mentioned above, very weak solutions; in this case an approximation procedure as considered in \cite{BG1, BG2} together with the a priori regularity estimates leads to assert the existence of a locally Lipschitz very weak solution $u \in W^{1,p-1}$. For the sake of brevity we shall not pursue this path here. As a matter of fact, all the a priori estimates obtained in the following do not necessitate \rif{VVV} but only that $V$ belongs to $L^2$. In turn, the need for considering $V \in L^2$ comes from the very fact that our results will be formulated via the use of a suitable non-linear potential of the function $|V|^2$, and that reveals to be a natural object when dealing with the boundedness of the gradient. More precisely we define the non-linear potential
\eqn{nonp}
$$
{\bf P}^V(x, R):= \int_0^R \left(\frac{|V|^2(B(x,\varrho))}{\varrho^{n-2}}\right)^{\frac{1}{2}}\, \frac{d\varrho}{\varrho} $$ where, here and for the rest of the paper, we use the more compact notation
$$
|V|^2(B(x,\varrho)):= \int_{B(x,\varrho)}|V(y)|^2\, dy\,.
$$
Estimates via non-linear potentials are a by now classical tool when studying the regularity of solutions of non-linear elliptic equations in divergence form: a classical one is the so called Wolff potential defined by
$$
{\bf W}_{\beta,p}^V(x, R):= \int_0^R \left[\frac{|V|(B(x,\varrho))}{\varrho^{n-\beta p}}\right]^{\frac{1}{p-1}}\, \frac{d \varrho}{\varrho} \qquad  \qquad \beta < n/p\,,
$$
where $V$ is a Radon measure.
In this respect pointwise estimates for solutions are obtained in \cite{KM, TW, mis2, DM3}. The peculiarity of the non-linear potential introduced in \rif{nonp} relies in that it allows to derive in a particularly favorable way several borderline estimates which seem unreachable otherwise.

The first results we present are about general local weak solutions to \rif{ppp} and indeed involve an $L^\infty$-estimate of the gradient via the non-linear potential ${\bf P}^V$.
\begin{theorem}\label{main01} Let $u \in W^{1,p}_{\loc}(\Omega)$ be a weak solution to the equation \trif{ppp0} for $p>1$ under the assumptions \trif{asp} and \trif{VVV}; there exists a constant $c$ depending only on $n,p,\ratio$ such that
\eqn{apl}
$$
\|Du\|_{L^\infty(B_{R/2})}\leq  c \left(\mean{B_R} (s+|Du|)^p\, dx \right)^{\frac{1}{p}}+ c\|{\bf P}^V(\cdot, R)\|_{L^\infty(B_{R})}^{\frac{1}{p-1}}
$$
holds whenever $B_{R}\subset \Omega$.
\end{theorem}
For more notation we refer to Section 2. An immediate corollary involves a sharp characterization of the Lipschitz continuity of solutions via the use of Lorentz spaces.
\begin{theorem}\label{mainu} Let $u \in W^{1,p}_{\loc}(\Omega)$ be a weak solution to the equation \trif{ppp0} for $p>1, n >2$, under the assumptions \trif{asp} and \trif{VVV}; assume that $V \in L(n,1)$ holds. Then $Du$ is locally bounded in $\Omega$. Moreover, for every open subset $\Omega' \Subset \Omega'' \Subset \Omega$ there exists a constant $c$ depending on $n, p, \ratio$ and dist$\,(\Omega', \partial\Omega)$ such that
\eqn{stimalo}
$$
\|Du\|_{L^\infty(\Omega')}\leq c \|Du\|_{L^p(\Omega'')} + c \|V\|_{L(n,1)(\Omega'')}+c s|\Omega''|^{1/p}\,.
$$
\end{theorem}
The interesting point in the previous result is that it shows that the classical, and optimal, Lipschitz regularity criterium for the equation $\triangle u=V$, that is indeed $V \in L(n,1)$ - a criterium that easily follows from the analysis of Riesz potentials - extends to the degenerate systems involving the $p$-Laplacean operator, independently of the values of $p$, and in regardless the degeneracy/singularity of the equations in question. Even more surprisingly, we shall see in a few lines how this result extends to the case of the $p$-Laplacean system. We refer to Section \ref{seclo} below for the relevant definitions concerning Lorentz spaces.

Next, we give analogous results for systems, that is when solutions are vector valued i.e. $N>1$. In this case it is necessary to introduce an additional structure condition, without which singularities of solutions in general appear and only partial regularity holds in general, no matter how regular the right hand side of the system is; see for instance \cite{dark}. The structure assumption in question is
\eqn{uh}
$$
 a(z)=g(|z|^2)z
$$
for a non-negative function $g\in C^1((0, \infty))$. Assumptions as \rif{uh} were first considered by Uhlenbeck in her seminal paper \cite{U}, and since then they have been used several times to rule out singularities of solutions.
\begin{theorem}\label{main02} Let $u \in W^{1,p}_{\loc}(\Omega, \er^N)$ be a weak solution to \trif{ppp} under the assumptions \trif{asp}, \trif{VVV} and \trif{uh}; there exists a constant $c$ depending only on $n,p,\ratio$ such that \trif{apl} holds whenever $B_{R}\subset \Omega$. In particular $Du$ is locally bounded when $V\in L(n,1)$ for $n>2$, and in this case estimate \trif{stimalo} holds.
\end{theorem}
As it will be immediately clear from the proof, in the previous statement we may just assume that $V\in L(n,1)$ locally in $\Omega$.

We now present a series of results concerning systems and equations whose model is given by the one in \rif{ppp0} with $q>0$. Of particular  interest for us is the one appearing in \rif{rivn}, already treated by Rivi\`ere in \cite{rivinv} for the case $n=2$; here we shall give a few result towards some problems exposed by Rivi\`ere in \cite{rivnotes}, and covering the systems in \rif{rivn} viewed as a particular case of the one in \rif{conf}. As usual in such cases - systems with right hand side having a potentially critical growth - we shall prove existence and regularity results for the associated Dirichlet problems. With $b \colon \Omega \times \er^N \times \ma \to \er$ being a Carath\'eodory functions satisfying
$$ |b(x,u,Du)|\leq (\Gamma+|Du|)^{q}\qquad \qquad 0 < q \leq p-1\,, \quad \Gamma \geq 0\,,
$$
we shall consider the Dirichlet problem
\eqn{ddd}
$$
\left\{
\begin{array}{ccc}
-\divo\, a(Du)= b(x,u,Du)V(x) & \mbox{in} & \Omega \\[3 pt]
 u= u_0 & \mbox{on}& \partial\Omega\,.
 \end{array}
 \right.
$$
Since our point is to prove local Lipchitz regularity of solutions, the regularity of the boundary data is not very relevant in this context; on the other hand we can nevertheless assume the following almost minimal assumptions:
\eqn{datum}
$$
u_0 \in W^{1,p_0}(\Omega) \qquad \mbox{for $p_0=p$ \ if \ $p< n$ \ and \ \ $p_0>n$ \ for \ $p=n$}\,.
$$
The first result we present is about the case $q=p-1$ that with some abuse of terminology will be called critical; see next section for further discussion.
\begin{theorem}[Critical case]\label{maind} Assume that $b(\cdot)$ satisfies \trif{aapp2} with $q =p-1$ and $p\leq n$. There exist constants $c_0,\ep_0>0$, depending only on $n,N,p,\ratio$ and $\Omega$, such that if
\eqn{smallii}
$$\|V\|_{L^n(\Omega)}< c_0 \quad \mbox{and} \quad \sup_{B(x,R)} {\bf P}^V(x,R)\leq \ep_0$$ hold,
then there exists a locally Lipschitz solution $u \in W^{1,p}(\Omega)$ to \trif{ddd}-\trif{datum}. Moreover there exists a constant $c$, depending only on $n,p,\ratio$, but otherwise independent of the solution $u$, of the vector fields $a(\cdot),b(\cdot)$, and on the datum $V(\cdot)$, such that whenever $B_{R}\subset \Omega$ it holds
\eqn{aes1}
$$
\|Du\|_{L^{\infty}(B_{R/2})}\leq c \left(\mean{B_R} (s+\Gamma+|Du|)^{p}\, dy \right)^\frac{1}{p} \,.
$$
\end{theorem}
Then we analyze the case $q < p-1$, again called subcritical with some abuse of terminology.
\begin{theorem}[Sub-critical case]\label{maind2} Assume that $b(\cdot)$ satisfies \trif{aapp2} with $0 < q <p-1$ and $p\leq n$; moreover, assume that
\eqn{smalliianc}
$$\|V\|_{L^2(\Omega)}+\|V\|_{L^\gamma(\Omega)}< \infty \quad \mbox{and} \quad \sup_{B(x,R)} {\bf P}^V(x,R)<\infty$$ hold with
$$
\gamma:=\frac{np}{np-nq-n+p}<n\,.
$$
Then there exists a locally Lipschitz solution $u \in W^{1,p}(\Omega,\er^N)$ to \trif{ddd}-\trif{datum}. Moreover there exists a constant $c$, depending only on $n,p,\ratio$, but otherwise independent of the solution $u$, of the vector fields $a(\cdot),b(\cdot)$, and on the datum $V(\cdot)$, such that whenever $B_{R}\subset \Omega$ it holds
\eqn{aes2}
$$
\|Du\|_{L^{\infty}(B_{R/2})}\leq c \left(\mean{B_R} (s+\Gamma+|Du|)^{p}\, dy\right)^\frac{1}{p}  + c \|{\bf P}^{V}(\cdot,R)\|_{L^\infty(B_R)}^{\frac{1}{p-q+1}}\,.
$$
\end{theorem}
Notice that formally letting $q\to 0$ in the previous result we obtain a problem of the type \rif{ppp} and indeed in this case $\gamma\to np/(np-n+p)$ when $q \to 0$, so that the assumption $\|V\|_{L^2(\Omega)}+\|V\|_{L^\gamma(\Omega)}< \infty$ in \rif{smalliianc} reduces to \rif{VVV}.

Finally, again the case of systems.
\begin{theorem}[Vectorial case $N>1$]\label{maind3} The results of Theorems \ref{maind}-\ref{maind2} remain valid when considering, in \trif{ddd}, an elliptic system satisfying the structure condition \trif{uh}. In particular, the existence and regularity results of Theorems \ref{maind}-\ref{maind2} are valid when considering solutions $u \in u_0 + W^{1,p}_0(\Omega, \er^N)$ to the $p$-Laplacean system
\eqn{dddv}
$$
\left\{
\begin{array}{ccc}
-\triangle_p u= b(x,u,Du)V(x) & \mbox{in}&\ \Omega \\[3 pt]
 u= u_0 & \mbox{on}&\ \partial\Omega\,,
 \end{array}
 \right.
$$
with $u_0$ as in \trif{datum}.
\end{theorem}
\begin{remark}[Criticality/subcriticality]\label{critic} In Theorem \ref{maind2} the smallness assumption on ${\bf P}^V(x,R)$ in \rif{smallii} can be obviously replaced by
\eqn{uniform}
$$
\lim_{R\to 0} \, \sup_{B(x,R)}\, {\bf P}^V(x,R)=0\,.
$$
This follows since the assumption on the smallness of ${\bf P}^V$ is only used to prove the validity of the local estimate \rif{aes1}, and therefore \rif{uniform} turns out to be sufficient after a standard localization process; see also Remark \ref{indepe} below.
In turn, as explained in the proof of Theorem \ref{mainu} - see in particular Section \ref{seclo} and \rif{mpotl} below - both the assumptions \rif{smalliianc} and in fact \rif{uniform} too can be simultaneously replaced by
\eqn{lloo}
$$ V \in L(n,1)\,.$$  Also note that $L^{n}$ is the largest space for which we can test equation \rif{critical} with $W^{1,p}$-functions still getting integrable quantities. For \rif{critical} the criticality is checked by using Young's inequality and Sobolev embedding theorem; indeed, for $p<n$ we have (use \rif{servedopo} below with $\gamma =n$)
$$
|Du|^{p-1}|V||u|\leq |u|^{\frac{np}{n-p}} + V^n + |Du|^p\,.
$$ The upgrade from $L^n$ to $L(n,1)$ gives the microscopical room sufficient for the gradient boundedness, thereby getting back the system in question from the realm of the critical problems to that of the subcritical ones. Note that any larger space of the type $L(n,\gamma)$ , $\gamma >1$ would not fit, and already in the case of the Laplacean operator.
\end{remark}
\begin{remark}\label{diffss} Since all our proofs are based on the growth conditions of the right hand side in \rif{ddd}$_1$, different structure conditions can be considered in order to cover the cases as in \rif{rivn}.
\end{remark}
\begin{remark}\label{compute}
The constant $c_0$ appearing in \rif{smallii} can be explicitly computed in terms of the constant occurring in the Sobolev embedding theorem on $\Omega$, that is $\|w\|_{L^{q^*}}\leq c(n,q,\Omega)\|Dw\|_{L^q}$ for $w \in W^{1,q}_0(\Omega)$, where $q^*$ denotes the Sobolev embedding inequality - we will choose $q \approx p$. This fact can be easily checked by careful tracing the constant dependence in the proofs in Section \ref{secap}.
\end{remark}
\begin{remark}\label{indepe} An interesting feature of the proof of Theorems \ref{maind} and \ref{maind2} is that the validity of the a priori estimates \rif{aes1}-\rif{aes2} is independent of the assumptions $\|V\|_{L^n(\Omega)}< c_0 $ and $\|V\|_{L^\gamma(\Omega)}< \infty $ in the following sense. We derive the a priori estimates \rif{aes1}-\rif{aes2} for solutions $u_\ep$ to certain {\em regularized problems} approximating the original one i.e. systems where the left hand side operator is smoother and non-degenerate, and such that the right hand side $b_\ep(\cdot)V_\ep$ is bounded both with respect with $Du$ and $x$. For such solutions the validity of \rif{aes1}-\rif{aes2} {\em is independent} of the assumptions $\|V\|_{L^n(\Omega)}< c_0 $ and $\|V\|_{L^n(\Omega)}< \infty $. Such assumptions are eventually needed in the final approximation process $\ep \to 0$.
\end{remark}
\begin{remark}\label{remarkt} From the proofs of the a priori estimates given in this paper it will be clear that in \rif{apl}, \rif{aes1} and \rif{aes2} the integral in the right hand side can be replaced by
$$\left(\mean{B_R} (s+\Gamma+|Du|)^{t}\, dy\right)^\frac{1}{t}$$
for every $t>0$; the constant $c$ then depends also on $t$ and blows-up for $t\to 0$.
\end{remark}
\subsection{Technical novelties}\label{tn} Let us here add a few comments about the techniques employed in this paper. We start with the a priori estimate \rif{apl} and Theorem \ref{main01}. The proof exploits in a suitable way a few hidden facts linked to the underlying property of $p$-harmonic maps $u$ to be such that functions of the type $|Du|^p$ are in turn sub-solutions to uniformly elliptic equations. This is commonly called Bernstein's trick and in the setting of degenerate problems its use goes back to the work of Uhlenbeck \cite{U}. In our context this fact cannot be used directly - since we are not dealing with homogeneous equations - but we will nevertheless take advantage of this fact in an indirect way, deriving Caccioppoli type inequalities with a suitable remainder term involving $|V|^2$ for the function
\eqn{bern}
$$v \thickapprox |Du|^p\,.$$ This is the content of Lemma \ref{caccio} below.
 According to the classical approach to regularity going back to De Giorgi \cite{DG}, this sole ingredient is then shown to be enough to prove the local boundedness of $v$, and therefore of the gradient $Du$; see Lemma \ref{degiorgii} below. This yields an a priori estimate involving the $L^{2p}$-norm of $Du$ and of the potential ${\bf P}^{\tilde V}$ in \rif{nonp} of a certain rescaled version $\tilde V$ of $V$; finally, a further iteration/interpolation scheme allows to derive the desired a priori estimate involving $Du$ in the natural energy space $L^p$. The whole procedure must be carried out for suitably regularized problems in order to allow for the use of certain quantities - like second derivatives of solutions - whose existence would not be otherwise guaranteed when considering general degenerate problems. As for the a priori estimate of Theorem \ref{main02}, we can follow the same approach of Theorem \ref{main01} once the Caccioppoli type inequality of Lemma \ref{caccio} is proved for functions as \rif{bern}; this is ultimately a consequence of the quasi-diagonal structure \rif{uh} in that it use allows for certain quantities to become controllable; see Section \ref{thmain02} below.

 The precise form of the estimate \rif{apl} - that is the use of the non-linear potential ${\bf P}^V$ - allows now for a rapid derivation of all the other a priori estimates in the paper. Indeed, as already observed in Remark \ref{critic}, the use of the potential ${\bf P}^V$ allows to establish in a sharp way the borderline spaces to which the right hand side must belong on order to make system critical or subcritical. As a matter of fact the proof of the Lorentz space criterium in Theorem \ref{mainu} follows directly by the property of ${\bf P}^V$; see Section \ref{seclo} below. As for the a priori estimates \rif{aes1}-\rif{aes2} these follow using \rif{apl} with $V$ replaced by $b(x,u,Du)V(x)$ and via the use suitable interpolation/interation inequalities; see Section \ref{apcomp} below and Remark \ref{indepe}. Once the a priori estimates for Theorems \ref{maind} and \ref{maind2} have been obtained the approximation scheme necessary to solve problem \rif{ddd} is finally built in Section \ref{secap}. Not surprisingly some delicate problems occur in the conformal case $p=n$ - which is related to the model problems \ref{conf} and \ref{rivn}, as remarked in the Introduction. The proof of the necessary a priori estimates involves indeed a deep result by Iwaniec \& Sbordone \cite{IS} on the rigidity of Hodge decomposition estimates under power type perturbations - see Theorem \ref{ISs} below. The final passage to the limit is then realized modifying some clever weak convergence arguments developed in \cite{DHM} in the context of measure data problems.

\section{Preliminaries}\label{ellipticse}
In this paper we follow the usual convention of denoting by $c$ a general constant larger (or equal) than one, possibly varying from line to line; special occurrences will be denoted by $c_1$ etc; relevant dependence on parameters will be emphasized using parentheses. We shall denote in a standard way $B(x_0,R):=\{x \in \er^n \, : \,  |x-x_0|< R\}$ the open ball with center $x_0$ and radius $R>0$; when not important, or clear from the context, we shall omit denoting the center as follows: $B_R \equiv B(x_0,R)$. Moreover, when more than one ball will come into the play, they will always share the same center unless otherwise stated. We shall also denote $B \equiv B_1 = B(0,1)$; more in general, when no confusion will arise or when the specific radius or center will not be important we shall abbreviate by $B$ any ball under consideration, while for $\alpha >0$ the symbol $\alpha B$ will denote the ball concentric to $B$ having radius magnified by the factor $\alpha$. With $A$ being a measurable subset with positive measure, and with $g \colon A \to \er^k$ being a measurable map, we shall denote its average by
 $$
    \mean{A}  g(x) \, dx  := \frac{1}{|A|}\int_{A}  g(x) \, dx\,.
$$
As it often happens when dealing with $p$-Laplacean type monotone operators, it will be convenient to work with a non-linear quantity involving the gradient rather than with the gradient itself, a quantity that takes into account the structure properties of the $p$-Laplacean operator. We define
\eqn{Vfun}
$$
W(z)  := |z|^{\frac{p-2}{2}}z\;, \qquad \qquad z
\in \er^{Nn}\;.
$$A basic property of the map $W(\cdot)$, whose proof can be found in
\cite[Lemma 2.1]{Ham}, is the following: For any $z_1, z_2 \in
\er^{Nn}$, and any $s\geq 0$, it holds, for $c\equiv c(n,p)$ \eqn{V}
$$
c^{-1}\Bigl( |z_1|^2+|z_2|^2 \Bigr)^{\frac{p-2}{2}}\leq
\frac{|W(z_2)-W(z_1)|^2}{|z_2-z_1|^2} \leq c\Bigl(
|z_1|^2+|z_2|^2 \Bigr)^{\frac{p-2}{2}}\;.
$$
From this last inequality it follows that $W(\cdot)$ is a locally bi-Lipschitz bijection of
$\er^n$ into itself.
The strict monotonicity properties of the vector field $a(\cdot)$ implied
by the left hand side in \rif{asp}$_1$ can be recast using the map
$W(\cdot)$. Indeed combining \rif{asp}$_2$ and \rif{V} yields, for $c\equiv
c(n,p,\nu)>0$, and whenever $z_1,z_2 \in \er^{Nn}$
\eqn{mon3}
$$   c^{-1}
|W(z_2)-W(z_1)|^2 \leq  \langle
a(z_2)-a(z_1),z_2-z_1\rangle\;.$$ When $p\geq 2$
the previous inequality immediately implies
\eqn{mon2}
$$   c^{-1}
|z_2-z_1|^p \leq  \langle a(z_2)-a(z_1),z_2-z_1\rangle\;.$$
We also notice that assumptions \rif{asp} together with a standard use of H\"odler's inequality, imply the existence of a constant $c\geq 1$ such that
\eqn{mony}
$$
|z|^p\leq c\langle a(z), z\rangle + cs^p\,,
$$
whenever $z \in  \er^{Nn}.$

Te next is a classical iteration lemma typically used in the regularity theory of variational equations.
\begin{lemma}\label{simpfun}\textnormal{(\cite[Chapter 6]{G})} Let $\varphi: [R/2, R]\to [0,\infty)$ be a bounded function such that the inequality
$$
\varphi(\varrho) \leq \frac{ \varphi(r)}{2} + \frac{\B}{(r-\varrho)^{\gamma}} + \A$$ holds whenever $R/2 <  \varrho < r <  R$,
for fixed constants $\A, \B, \gamma\geq 0$.
Then it holds that $$ \varphi(R/2) \leq
\frac{c\B}{R^{\gamma}} + c\A$$
for a suitable constant $c$ depending only on $\gamma$.
\end{lemma}
We recall the following fundamental rigidity theorem of the Hodge decomposition with respect to power perturbations; see \cite{I, IS}.

\begin{theorem}\textnormal{(Iwaniec \& Sbordone \cite[Theorem 3]{IS})}\label{ISs} Let $\Omega\subset \er^n$ be a regular domain, and let $w \in W^{1,t}_0(\Omega, \er^N)$, with $t>1$ and $N \geq 1$; let $\delta \in (-1,t-1)$. Then there exists a vector field $\varphi \in  W^{1,\frac{t}{1+\delta}}_0(\Omega, \er^N)$ and a divergence free matrix field $H \in L^{\frac{t}{1+\delta}}(\Omega, \ma)$ such that
$$
|Dw|^\delta Dw =D\varphi + H
$$
and
$$
\|H\|_{L^\frac{t}{1+\delta}(\Omega)}\leq c~\delta\|Dw\|_{L^t(\Omega)}^{1+\delta}
$$
hold, for a constant $c$ depending only on $n,\Omega$.
\end{theorem}
The  previous result will not be used here to derive a priori estimates for solutions, but it will rather be employed to perform the approximation procedure necessary proving Theorems \ref{maind2}-\ref{maind3} in the conformal case $p=n$; see Section \ref{secap} below.

The next is a suitable, boundary version of Gehring's lemma, which actually also holds under less restrictive assumptions that those considered here. The statement below will anyway suite our purposes, the main emphasis here being on the stability of the exponents and their correct dependence upon the various constants.
\begin{theorem}\label{gehringt} Let $v \in W^{1,p}(\Omega,\er^N)$ be a solution to the Dirichlet problem
$$
\left\{
\begin{array}{ccc}
-\divo\, \tilde{a}(Dv)= g & \mbox{in} & \Omega \\[3 pt]
 v\equiv v_0 & \mbox{on}& \partial\Omega\,,
 \end{array}
 \right.
$$
where the vector field $\tilde a \colon \Omega \to \ma$ satisfies assumptions \trif{asp}, $g \in L^\infty(\Omega, \er^N)$ and $u_0 \in W^{1,p_0}(\Omega, \er^N)$ for some $p_0 > p$; here $\Omega$ is a Lipschitz domain. Then there exists an exponent $p_1$,
\eqn{gehring}
$$
p< p_1 \equiv p_1\left(n,N,\ratio,[\partial \Omega]_{C^{0,1}}\right) \leq p_0\,,
$$
independent of the solution considered $v$, of the boundary datum $v_0$, of the vector field $\tilde a(\cdot)$, and of the function $g$, such that $v \in W^{1,p_1}(\Omega, \er^N)$.\end{theorem}
In the way stated above the result can be deduced from the stronger results available in the literature, see for instance in \cite{Ark, DKM, dark, KMe}. For later applications we remark that the main point of interest here is that the exponent in \rif{gehring} is independent of the size of the norm $\|g\|_{L^\infty}$. Instead, as expected, the only thing blowing-up up when $\|g\|_{L^\infty}\to \infty$ is the constant $c$ appearing in the a priori estimate for $\|Du\|_{L^{p_0}}$ associated to the previous result, and that is indeed not reported here since it is not going to be used in the following.
\subsection{Smoothing}\label{smoothing} We finally
recall a few facts linked to some standard regularization methods; let us fix a family $\{\phi_\ep\}_{\ep>0}$ of standard mollifiers in $\ma$ and obtained in the following way: $\phi_\ep(z):=\ep^{-Nn}\phi(z/\ep)$. Here $\phi \in C^{\infty}(\er^{Nn})$ and it is such that
\eqn{suppmoll}
$$
\textnormal{supp}\, \phi=\overline{B_1}\qquad \mbox{and}\qquad \int_{\ma} \phi(z)\, dz =1\,.$$ We define the regularized vector fields
\eqn{reg}
$$
a_\ep (z):= (a*\phi_\ep)(z)\, ,
 \qquad \ep>0\, .
$$
It obviously follows that $a_\ep
(\cdot) \in C^\infty(\ma) $ and moreover, as in \cite[Lemma 3.1]{ELM} - whose arguments apply here since \rif{suppmoll} is assumed - we have that the following growth and ellipticity properties are satisfied for constants $\nu_0, L_0$ depending on $n,N,\nu,L,p$ but otherwise independent of $\ep$:
\eqn{aspep}
$$
\left\{
    \begin{array}{c}
    |a_\ep(z)|+|\partial a_\ep(z)|(|z|^2+s_\ep^2)^{\frac{1}{2}} \leq L_0(|z|^2+s_\ep^2)^{\frac{p-1}{2}} \\ [3 pt]
    \nu^{-1}_0(|z|^2+s_\ep^2)^{\frac{p-2}{2}}|\lambda|^{2} \leq \langle \partial a_\ep(z)\lambda, \lambda
    \rangle
    \end{array}
    \right.
$$
whenever $z, \lambda \in \ma$,
where
\eqn{se}
$$
s_\ep := s+ \ep>0\,.
$$
We recall that by the very definition in \rif{reg} we have that
\eqn{unic}
$$
a_\ep \to a \qquad \mbox{uniformly on compact subsets of} \ \ma\,.
$$
We now discuss the smoothing procedure when dealing with vector fields $a(\cdot)$ under the additional structure assumption \rif{uh}; the approximation should be done in order to preserve \rif{uh} so that we shall build the approximation starting from \rif{uh} - see also \cite[Lemma 3.2]{ELM}. For $\ep>0$ we define
\eqn{appvec}
$$
a_\ep(z):= g_{\ep}(|z|^2)z\,,\qquad \mbox{where}\ \ g_\ep(t):= g(\ep^2+t)\,.
$$
Using \rif{asp} an \rif{uh} it is now easy to see that the same \rif{asp} and \rif{unic} are satisfied with $s_\ep$ defined as in \rif{se}. Observe that, in particular \rif{aspep} are satisfied with $\nu_0 =\nu$ and $L=L_0$.
\section{Theorem \ref{main01}}\label{secth}
The proof falls is divided in four steps, and follows the path outlined in Section \ref{tn}. Moreover, since the result is local in nature we may assume with no loss of generality that $u \in W^{1,p}(\Omega)$.

{\em Step 1: Approximation}. Since the statement is local in nature and so is estimate \rif{apl}, we shall prove Theorem \ref{main01} replacing $\Omega$ by any open subset $\Omega' \Subset \Omega$ with smooth boundary; therefore the ball $B_R$ considered in \rif{apl} will be such that $B_R \Subset \Omega'$. In the same way all the balls considered in the rest of the proof will be contained in $\Omega'$. We first regularize the left hand side vector field according to \rif{reg} and then the right hand side potential as
\eqn{truncation}
$$
V_\ep(x):=\max\{\min\{V(x),1/\ep\}, -1/\ep\}\,.
$$ We define $u_\ep \in u + W^{1,p}_0(\Omega')$
 as the unique solution to the following Dirichlet problem:
\eqn{Dirapp}
$$
\left\{
    \begin{array}{ccc}
    -\divo \ a_\ep(Du_\ep)=V_\ep & \qquad \mbox{in} & \Omega'\\
        u_\ep= u&\qquad \mbox{on}& \partial\Omega'.
\end{array}\right.
$$
By \rif{aspep} standard regularity theory applies  - see Section \ref{serve} below - and we have
\eqn{serve}
$$
u_\ep
 \in W^{2,2}_{\loc}(\Omega')\cap C^{1, \alpha}_{\loc}(\Omega')
$$
for some $\alpha \in (0,1)$ depending on $\ep$, and therefore, considering the function
\eqn{subso}
$$
v_\ep:= (s_\ep^2 + |Du_\ep|^2)^{\frac{p}{2}}\,,
$$
we observe that \rif{serve} implies
\eqn{serve2}
$$
v_\ep
\in W^{1,2}_{\loc}(\Omega')\cap C^{0, \alpha}_{\loc}(\Omega')\,.
$$ Note that formally it is $v\equiv |Du|^p$ when $s_\ep=0$. We shall derive energy estimates for $v_\ep$, which are uniform with respect to $\ep$. Moreover, in the following we shall use the shorthand notation $v \equiv v_\ep$, $u\equiv  u_\ep$, $V \equiv V_\ep$ and $a_\ep(Du_\ep) \equiv a(Du)$, recovering the full notation from time to time in order to emphasize the uniformicity of the estimates with respect to $\ep$. Finally we shall denote $\Omega'\equiv \Omega$.

{\em Step 2: A Caccioppoli type inequality}. Here we show that the functions $v_\ep$ defined in \rif{subso} satisfy a suitable energy inequality - of so called Caccioppoli type - that in turn will imply the desired pointwise estimate for the gradient. Specifically, we have
\begin{lemma}[Caccioppoli's inequality for $v_\ep$]\label{caccio} Let $v_\ep\equiv v$ be the function defined in \trif{subso}; there exists a constant $c_1$, which only depends on $n,p,\ratio$, but is otherwise independent of $\ep$, such that whenever $B_R \subset \Omega$ is a ball with radius $R$ it holds that
\eqn{cacci}
$$\int_{B_{R/2}} |D(v-k)_+|^2  \, dx  \leq  \frac{c_1}{R^2} \int_{B_{R}} (v-k)_+^2 \, dx + c_1\int_{B_R}|\tilde V|^2\, dx  \,,
$$
where
\eqn{tildev}
$$\tilde V \equiv \tilde V_\ep := \left( s_\ep^2+\|Du_\ep\|_{L^{\infty}(B_{R})}^2\right)^{\frac{1}{2}}V_\ep\,.$$
\end{lemma}
\begin{proof}
In the following we shall keep on using Einstein's convention on repeated indexes. In the weak formulation of \rif{Dirapp}$_1$
\eqn{sys1}
$$
\int_{\Omega} \langle a(Du),D\varphi\rangle \, dx =\int_{\Omega}  a_i(Du)D_i\varphi\, dx   = \int_{\Omega} \varphi V \, d x\,,
$$
we use $D_s \varphi$ instead of $\varphi$ for $s \in \{1,\ldots,n\}$ as testing function;
integration by parts then yields
\eqn{l1}
 $$
\int_{\Omega} \partial_{z_j}a_i(Du)D_jD_su D_i\varphi \, dx = -\int_{\Omega} D_s\varphi V \, dx \,.
$$
Then we introduce
$$
\tilde{a}_{i,j}(x):= \frac{\partial_{z_j}a_i(Du(x))}{(s_\ep^2 + |Du(x)|^2)^{\frac{p-2}{2}}}\,,
$$
which is a bounded and uniformly elliptic matrix in view of \rif{asp}; in view of \rif{aspep} this means that
\eqn{elli}
$$
|\tilde{a}_{i,j}(x)|\leq L_0\, \qquad \qquad \nu_0|\lambda|^2 \leq \langle \tilde{a}(x)\lambda, \lambda \rangle
$$
for every $x \in \Omega$ and $\lambda \in \er^n$, where as in \rif{aspep} it holds that $0 <\nu_0\leq L_0 < \infty$; $\nu_0, L_0$ are independent of $\ep$, and only depend on $n,p,\ratio$. In the following, for $z \in \er^n$, we shall also let
$$
H_\ep(z):=(s_\ep^2 + |z|^2)^{\frac{1}{2}}\,, \qquad \qquad
H(Du)\equiv H_\ep(Du_\ep):=(s_\ep^2 + |Du_\ep|^2)^{\frac{1}{2}}\,.
$$
With such a notation we have
$H(Du)^p=v$ and \rif{l1} reads
\eqn{ee1}
 $$
\int_{\Omega} \tilde{a}_{i,j}H(Du)^{p-2} D_jD_su D_i\varphi \, dx = -\int_{\Omega} D_s\varphi V \, dx \,.
$$
Note that by \rif{serve} and \rif{elli} a standard density argument this last equation remains valid whenever $\varphi \in W^{1,2}(\Omega)$ has compact support in $\Omega$. Therefore, with $\eta \in C^{\infty}_0(B_R)$ being a non-negative cut-off function, in \rif{ee1} we may use the test function
$$
\varphi:=\eta^2          (v-k)_+ D_su
$$
Note that in view of \rif{serve}-\rif{serve2} this is an admissible test function in \rif{ee1}. Then we observe that, for $h \in \{1,\ldots,n\}$, we have
\eqn{deri}
$$
D_h\varphi = \eta^2(v-k)_+D_hD_su   + \eta^2D_su D_h(v-k)_+ + 2\eta (v-k)_+D_h\eta D_su \,,
$$ and therefore, summing up over $s\in \{1,\ldots,n\}$ yields
\begin{eqnarray}
 \nonumber I_1 + I_2 + I_3 &:=&  \int_{\Omega} \eta^2\tilde{a}_{i,j}H(Du)^{p-2} D_jD_su D_iD_su (v-k)_+ \, dx \\
  \nonumber && \qquad   + \int_{\Omega} \eta^2\tilde{a}_{i,j}H(Du)^{p-2} D_jD_su D_su D_i(v-k)_+ \, dx\\
  \nonumber&& \qquad   + 2\int_{\Omega} \eta\tilde{a}_{i,j}H(Du)^{p-2} D_jD_su  D_su (v-k)_+D_i\eta \, dx\\
  \nonumber  \qquad   &:=& - \int_{\Omega}  \eta^2D_sD_su (v-k)_+ V\, dx - \int_{\Omega}  \eta^2D_su D_s(v-k)_+V \, dx \\
   && \qquad   - 2\int_{\Omega} \eta D_s\eta D_su (v-k)_+V\, dx =: II_1+ II_2+ II_3\,.\label{bbb}
\end{eqnarray}
The terms appearing
in the left hand side of \rif{bbb} can be in turn estimated via Young's inequality and ellipticity as follows. Keeping \rif{subso} in mind and noticing that
$$
D_j v =p H^{p-2}(Du) D_jD_suD_su
$$
we have that
$$
\int_{\Omega} \eta^2\tilde{a}_{i,j}D_j (v-k)_+D_i (v-k)_+ \, dx =p I_2
$$
and therefore \rif{elli} yields
$$
 \nu_0\int_{\Omega} |D (v-k)_+|^2\eta^2 \, dx \leq pI_2\,.
$$
Again using ellipticity we have
\eqn{i5}
$$
I_5:= \int_{\Omega} \eta^2H^{p-2}(Du)|D^2u|^2 (v-k)_+ \, dx\leq  \nu_0^{-1} I_1\,.
$$
Finally, using Young's inequality with $\delta \in (0,1)$, we gain
\begin{eqnarray*}
I_3 &\leq & c\int_{\Omega} \eta|D\eta| |D(v-k)_+| (v-k)_+ \, dx\\
&\leq &  \delta\int_{\Omega} |D (v-k)_+|^2\eta^2 + c(\delta)
\int_{\Omega} (v-k)_+^2|D\eta|^2\, dx \,.
\end{eqnarray*}
By choosing $\delta\equiv \delta (n,p,\ratio)$ small enough in order to re-absorb
 terms we arrive at
\eqn{inter}
$$I_5 +  \int_{\Omega} \eta^2|D (v-k)_+|^2 \, dx\leq c\int_{\Omega} |D \eta|^2(v-k)_+^2 \, dx + c\big| II_1+II_2+III_3\big|\,,
$$
with $c$ depending only on $n,p,\ratio$.
As for the right hand side terms we shall simply initially estimate, everywhere inside the integrals, as follows:
\eqn{naive}
$$
|D_su|\leq \|Du\|_{L^{\infty}}\equiv \|Du\|_{L^{\infty}({\rm supp}\, \eta)}\,.
$$
Now, recalling the definition of $v$, we have that whenever we evaluate functions on the set supp $\eta$ it holds
$$
(v-k)_+ \leq \sqrt{(v-k)_+}H(Du)^{\frac{p}{2}}\leq ( s_\ep^2+\|Du\|_{L^{\infty}}^2)^{\frac{1}{2}}\sqrt{(v-k)_+}H(Du)^{\frac{p-2}{2}}\,,
$$
and therefore, using Young's inequality with $\delta \in (0,1)$, we further have
\begin{eqnarray}
 |II_1| &\leq & c\int_{\Omega}\eta^2 |D^2u| (v-k)_+ |V|\, dx \nonumber \\ &\leq & c( s_\ep^2+\|Du\|_{L^{\infty}}^2)^{\frac{1}{2}}\int_{\Omega}\eta^2 [(v-k)_+]^\frac{1}{2}H(Du)^{\frac{p-2}{2}} |D^2u||V |\, dx\nonumber \\
 & \leq &\delta I_5+ c( s_\ep^2+\|Du\|_{L^{\infty}}^2)\int_{\Omega}\eta^2 |V |^2\, dx\,.\label{IIuno}
\end{eqnarray}
Again by Young's inequality we have, with $\delta \in (0,1)$, that
\eqn{IIdue}
$$
|II_2| \leq \delta \int_{\Omega} \eta^2|D(v-k)_+|^2\, dx  + c(\delta)
\|Du\|_{L^{\infty}}^2 \int_{\Omega}\eta^2  |V|^2\, dx\,,$$
and
\eqn{pdopo}
$$ |II_3| \leq c\int_{\Omega} |D\eta|^2(v-k)_+^2\, dx + c \|Du\|_{L^{\infty}}^2 \int_{\Omega} \eta^2 |V|^2\, dx\,. $$
Matching all the previous inequalities, choosing $\delta$ small enough in order to re-absorb the terms involving $\delta$ in the right hand side, and finally taking $\eta$ in a standard way - i.e. $\eta \equiv 1$ on $B_{R/2}$, $0 \leq \eta \leq 1$, and $\|D \eta\|_{L^\infty}\leq cR^{-1}$ - we arrive at the Caccioppoli's inequality \rif{cacci}.
\end{proof}
{\em Step 3: An oscillation improvement estimate.} Here we prove the following:
\begin{lemma}\label{itera} Let $(v,\tilde V)$ be a couple of functions defined in an open subset $\Omega$, with $v \geq 0$, and such that \trif{cacci} holds whenever $k\geq 0$, for a fixed $B \equiv B_{R} \subset \Omega$; moreover let $k,d>0$ be fixed and assume that
\eqn{start}
$$
|(1/2)B \cap \{v >k\}|\leq  \frac{1}{d^2}\int_{(1/2)B } (v-k)_+^2 \, dx
$$
holds. Then there exists a constant $c \equiv c(n,c_1)$
such that there holds
\begin{align}
    \bigg(\frac{1}{d^2  R^n} \int_{(1/2)B}&(v-k)_+^2 \, dx\bigg)^{\frac{\chi}{2}}    \nonumber \\
    & \leq c\left(\frac{1}{d^2  R^n} \int_{B} (v-k)_+^2 \, dx\right)^{\frac{1}{2}}
        + c\left(\frac{|\tilde V|^2(B)}{d^2R^{n-2}}\right)^{\frac{1}{2}}\, , \label{disit}
\end{align}
where $\chi \in (0,1)$ depends only on $n$.
\end{lemma}
\begin{proof}Let $B\equiv B(x_0,R)$ for some $x_0 \in \Omega'$ be the ball in question.
We start observing that we can reduce to the case
$
B \equiv B_1$ and $d=1$. Indeed, let us define
$$
w(y):= \frac{v(x_0+Ry)}{d}\,, \qquad W(y):= \frac{R\tilde V(x_0+Ry)}{d}\,,\qquad y \in B_1\,.
$$
The couple $(w, W)$ satisfies inequality \rif{cacci} with $(B_R, d,k)$ replaced by $(B_1, 1, k/d)$;
moreover we observe that also \rif{start} is fulfilled for $w$
with $k$ replaced by $k/d$. Therefore it suffices to prove the lemma for $(w, W)$, i.e. , i.e. inequality \rif{disit}, on $B\equiv B_1$ where $k$ is replaced by $k/d$; the
general case follows scaling back to $(v, \tilde V)$.
Therefore we pass to the proof for the special situation considered above assuming that
\eqn{cacci1}
$$\int_{B_{1/2}} |D(v-k)_+|^2  \, dx  \leq  c_1 \int_{B_1} (v-k)_+^2 \, dx + c_1\int_{B_1}|\tilde V|^2\, dx  \,.
$$
Moreover we assume that
$
|B_1 \cap \{v >k\}|>0,
$
otherwise \rif{disit} trivializes. Similarly, we may also assume that
$$
    \int_{B_1} (v-k)_+^2 \, dx>0\,.
$$
Now we denote by
$$
2<t:=\left\{
\begin{array}{ccc}
 \frac{2n}{n-2} &\mbox{if}& n>2\\[3 pt]
\mbox{any number larger than two}&\mbox{if}& n=2
\end{array}
\right.
$$
the usual Sobolev embedding exponent and apply Sobolev embedding theorem
$$
\left(\int_{B_{1/2}} (v-k)_+^t \, dx\right)^{\frac{2}{t}}\leq c \int_{B_{1/2}} |D(v-k)_+|^2  \, dx + c\int_{B_{1/2}} (v-k)_+^2 \, dx\,.
$$
Together with \rif{cacci1} the last estimate now yields the following reverse H\"older-type
inequality:
\eqn{cacci3}
$$
\left(\int_{B_{1/2}} (v-k)_+^t \, dx\right)^{\frac{2}{t}}\leq   c\int_{B_{1}} (v-k)_+^2 \, dx+ c\int_{B_1}|\tilde V|^2\, dx  \,.
$$
On the other hand by \rif{start} we have
\begin{eqnarray*}
\left(\int_{B_{1/2}}(v-k)_+^2 \, dx\right)^{\frac{2}{t}} &= &\left(\int_{B_{1/2}} (v-k)_+^2 \, dx\right)^{\frac{2}{t}-1}\int_{B_{1/2}} (v-k)_+^2 \, dx\\
&\leq & |B_{1/2} \cap \{v >k\}|^{\frac{2}{t}-1}\int_{B_{1/2}} (v-k)_+^2 \, dx\,.
\end{eqnarray*}
In turn, applying H\"older's inequality yields
$$\int_{B_{1/2}} (v-k)_+^2 \, dx \leq |B_{1/2} \cap \{v >k\}|^{1-\frac{2}{t}}\left(\int_{B_{1/2}} (v-k)_+^{t} \, dx\right)^{\frac{2}{t}}\,.$$
The last two inequalities give now
\eqn{ca1}
$$
\left(\int_{B_{1/2}} (v-k)_+^{2} \, dx\right)^{\frac{2}{t}}\leq \left(\int_{B_{1/2}} (v-k)_+^{t} \, dx\right)^{\frac{2}{t}}\,.
$$
Inequality \rif{disit}
now follows estimating the right hand side of \rif{ca1} by mean of \rif{cacci3}, taking of course $\chi = 2/t \in (0,1)$.
\end{proof}
{\em Step 3: Iteration and a priori estimate for $v$.} Here we iterate Lemma \ref{itera} - applied in the context of Step 1 - in order to get a first pointwise estimate of $v$ in terms of $\tilde V$. More precisely, we prove the following abstract result, that will be eventually applied with the choice $v \equiv v_\ep$ (here we are slightly abusing the notation) and for $\tilde V$ defined in \rif{tildev}.
\begin{lemma}[Abstract potential estimate]\label{degiorgii} Let $v \in W^{1,2}_{\loc}(\Omega)$ be a function satisfying \trif{cacci} whenever $B_R \subset \Omega'$, for a certain function $\tilde V \in L^2(\Omega')$  and for every $k,d>0$. Then there exists a constant $c$, depending only on $n$ and $c_1$, such that
\eqn{app1x}
$$
|v(x)|\leq c \left(\mean{B(x,R)} |v|^2\, dy \right)^{\frac{1}{2}} + c{\bf P}^{\tilde{V}}(x,2R)\,,
$$
holds for every ball $B(x, 2R) \subset \Omega'$.
\end{lemma}
\begin{proof} We shall use a suitable modifications of De Giorgi's iteration technique \cite{DG} following the approaches proposed in \cite{KM}. We consider a dyadic sequence of radii $R_j:=2^{1-j}R$, $j \in \en$, and let $B_j := B(x,R_j)$. We define
$k_0:=0$, and, recursively for $j \geq 0$,
\eqn{inter0}
$$
k_{j+1}:= k_j + \left(\frac{1}{\delta^2 R_j^n}\int_{B_{j+1}} (v-k_j)^2_+\, dx \right)^{\frac{1}{2}}\,.
$$
The number $\delta \in (0,1)$ will be chosen in a few lines; it is going to be suitably small, but will be always chosen
in a way making it depending only on $n$ and $c_1$.
We note that $\{k_j\}$ is a non-decreasing sequence.
We will prove the validity of the estimate
\eqn{decay}
$$
k_{j+1} -k_j  \leq \frac{1}{2}\left(k_j -k_{j-1} \right)+ c
  \left(\frac{|\tilde V|^2(B_j)}{R_j^{n-2}}\right)^{\frac{1}{2}}
$$
for every $j \in \en$, and for a constant $c$ only depending on $n,c_1$. In order to prove \rif{decay} for $j \in \en$ we preliminary observe that we may assume
$k_{j+1}>k_j$;
otherwise \rif{decay} itself is trivially satisfied.
Then we start showing that
\eqn{absorb}
$$
\delta^{\chi}\leq c \,\delta\,\frac{k_j-k_{j-1}}{k_{j+1}-k_j}+  c\left(\frac{|\tilde V|^2(B_j)}{d_j^2R_j^{n-2}}\right)^{\frac{1}{2}}
$$
holds for every $j \in \en$, with a constant $c$ depending only on $n,c_1$, and where
\eqn{defid}
$$
d_j :=\frac{k_{j+1}-k_{j}}{2^{n/2}}\,.
$$
Here $\chi \in (0,1)$ is the number introduced in Lemma \ref{itera}. By mean of the definition given in \rif{inter0} we have
\begin{eqnarray*}
|B_j\cap \{v> k_j\}|  & \leq & \frac{1}{(k_j-k_{j-1})^2}\int_{B_{j}\cap \{v> k_j\}}(v-k_{j-1})_+^2\, dx\\  &\leq & \frac{1}{(k_j-k_{j-1})^2}\int_{B_{j}} (v-k_{j-1})_+^2\, dx\\&= &  \delta^2 R_{j-1}^n=
2^n\delta^2 R_{j}^n= \frac{2^n}{(k_{j+1}-k_{j})^2}\int_{B_{j+1}} (v-k_{j})_+^2\, dx\,.
\end{eqnarray*}
Observe that we repeatedly used \rif{inter0}. From the last chain of inequalities it is clear that choosing $\delta$ small enough in order to have $
\delta^2 \leq 2^{-(n+1)}R_j^{-n}|B_j|
$
- which imposes on $\delta$ a smallness condition depending only on $n$ - we obtain
\eqn{mezza}
$$
|B_j\cap \{v> k_j\}|\leq \frac{1}{2}|B_j|\,.
$$
With the definition in \rif{defid} we are able to apply Lemma \ref{itera} for the choices
$k \equiv k_j$, $d \equiv d_j$ and $B \equiv B_{j+1}$,
as assumption \rif{start} turns out to be satisfied. This yields
\begin{align}
\nonumber \bigg(
 \frac{1}{d_j^2   R_j^n} \int_{B_{j+1} }& (v-k_{j})_+^2 \, dx\bigg)^{\frac{\chi}{2}}\\
&
\leq c
\left(\frac{1}{d_j^2   R_j^n} \int_{B_{j} } (v-k_{j})_+^2 \, dx\right)^{\frac{1}{2}}+ c\left(\frac{|\tilde V|^2(B_j)}{d_j^2R_j^{n-2}}\right)^{\frac{1}{2}}\label{disit2}\;,
\end{align}
where $\chi \in (0,1)$ is the number introduced in Lemma \ref{itera}.
In turn, using \rif{inter0} again we observe that
$$\left(\frac{1}{d_j^2   R_j^n} \int_{B_{j} } (v-k_{j})_+^2 \, dx\right)^{\frac{1}{2}} \leq
\left(\frac{1}{d_j^2   R_j^n} \int_{B_{j} } (v-k_{j-1})_+^2 \, dx\right)^{\frac{1}{2}} =  2^n\delta \left(\frac{d_{j-1}}{d_j}\right) \,.
$$
Merging the last inequality with \rif{disit2}, and using \rif{inter0} together with the definition of $d_j$ in \rif{defid}, we find
\begin{eqnarray}
\delta^{\chi}&\leq& c \left(
 \frac{1}{d_j^2   R_j^n} \int_{B_{j+1}} (v-k_{j})_+^2 \, dx\right)^{\frac{\chi}{2}}\nonumber\\
 &\leq&
 c
\left(\frac{1}{d_j^2   R_j^n} \int_{B_{j} } (v-k_{j})_+^2 \, dx\right)^{\frac{1}{2}}+ c\left(\frac{|\tilde V|^2(B_j)}{d_j^2R_j^{n-2}}\right)^{\frac{1}{2}}\nonumber\\
 &\leq &
c \, \delta\,\frac{k_{j}-k_{j-1}}{k_{j+1}-k_{j}} +  c\left(\frac{|\tilde V|^2(B_j)}{d_j^2R_j^{n-2}}\right)^{\frac{1}{2}}\,,\label{lunga}
\end{eqnarray}
for a constant $c$ which depends only on $n,c_1$; the proof of \rif{absorb} is therefore complete. We can now show that \rif{decay} holds.
Indeed, if $
k_{j+1} -k_j  \leq (1/2)(k_j -k_{j-1})
$
holds so also \rif{decay} does, trivially. Otherwise, we have
$
(k_j -k_{j-1})/(k_{j+1} -k_j )< 2,
$
which, used in \rif{lunga}, yields
$$
\delta^{\chi}\leq \tilde{c}\, \delta+c
\left(\frac{|\tilde V|^2(B_j)}{d_j^2R_j^{n-2}}\right)^{\frac{1}{2}}\,,
$$
with $\tilde{c}$ depending only on $n,c_1$. Therefore reducing further the size of $\delta \equiv \delta (n,c_1)$ in order to have
$\delta  < (1/2\tilde{c})^{1/(1-\chi)}$, and recalling the choice of $d_j$, we conclude with
$$
k_{j+1}-k_j\leq c
 \left(\frac{|\tilde V|^2(B_j)}{R_j^{n-2}}\right)^{\frac{1}{2}}
$$ so that \rif{decay} follows in any case. The proof of \rif{app1x} can now be obtained iterating \rif{decay}:
\begin{eqnarray}
\nonumber k_{m}-k_1 & \leq & k_{m+1}-k_1= \sum_{j=1}^{m} (k_{j+1} -k_j )\\\nonumber &\leq &   \frac{1}{2} \sum_{j=1}^{m} (k_{j}-k_{j-1} ) + c \sum_{j=1}^{m} \left(\frac{|\tilde V|^2(B_j)}{R_j^{n-2}}\right)^{\frac{1}{2}} \\ & \leq  &\frac{1}{2} k_{m} + c\sum_{j=1}^{m}
\left(\frac{|\tilde V|^2(B_j)}{R_j^{n-2}}\right)^{\frac{1}{2}}\leq   \frac{1}{2} k_{m}+ c\,{\bf P}^{\tilde V}(x,2R)
\label{discrete}
\,,\end{eqnarray}
where we used the content of Remark \ref{discrete2} below and that $k_0=0$ by definition.
Therefore we have that
$$
\lim_{m \to \infty} k_m  \leq 2k_1 +c\,{\bf P}^{\tilde V}(x,2R).
$$
On the other hand, since $v$ is continuous - see \rif{serve2} - and since \rif{mezza} implies that $\inf_{B_m} v \leq k_m$
we have
\eqn{quasi}
$$
|v (x)|  =  \lim_{m \to \infty} \, \inf_{B_m}\, v\leq  \lim_{m \to \infty} k_m\nonumber \leq  2k_1 +c\,{\bf P}^{\tilde V}(x,2R)
$$
where $c$ depends only $n,c_1$. At this point estimate \rif{app1x} follows taking the definition of $k_1$ into account. The proof of Lemma \ref{degiorgii} is complete. \end{proof}
At this point we can apply both Lemmas \ref{itera} and \ref{degiorgii} to get that the following inequality:
\eqn{app1}
$$
|v_\ep(x)|\leq c \left(\mean{B(x,R)} |v_\ep|^2\, dy \right)^{\frac{1}{2}} + c\,{\bf P}^{\tilde{V}\ep}(x,2R)\,,
$$
holds for every ball $B(x, 2R) \subset \Omega'$, for a constant $c$ depending only on $n,p,\ratio$, but otherwise independent of $\ep$. It goes without saying that the function \rif{tildev} is now the one defined in \rif{tildev} while $v_\ep$ has been introduced in \rif{subso}.

{\em Step 4: A priori estimate for $Du$.} We now derive the a priori estimate for $\|Du\|_{L^\infty}\equiv \|Du_\ep\|_{L^\infty}$, which will be of course uniform with respect to $\ep>0$. In terms of $Du\equiv Du_\ep$ and $V\equiv V_\ep$ estimate \rif{app1} amounts to have
\begin{eqnarray}
\nonumber (s_\ep^2+|Du(x)|^2)^{\frac{p}{2}}&\leq &c\left(\mean{B(x,\tr)} (s_\ep^2+|Du|^2)^{p}\, dy \right)^{\frac{1}{2}}\\&&\qquad +c\left( s_\ep^2+\|Du\|_{L^{\infty}(B(x,\tr))}^2\right)^{\frac{1}{2}}{\bf P}^{V}(x,\tr)\label{cover1}
\end{eqnarray}
whenever $B(x, \tr) \subset \Omega'$. Note that since the radii considered are arbitrary in \rif{app1}, upon enlarging the constants involved by
a factor essentially depending on $n$, we may avoid considering ${\bf P}^{\tilde V}(x,2\tr)$ and use ${\bf P}^{\tilde V}(x,\tr)$. We now fix a ball $B_{R}
\subset \Omega$ and related concentric balls $B_{R /2} \subset B_\varrho\subset B_r \subset B_{R}$ with $R/2 < \varrho \leq r < R$. We use \rif{cover1} with $x \in B_{\varrho}$ and with $\tr = r-\varrho$ in such a way that $B(x,\tr) \subset B_{r}$; at this stage we recall that $V$ is defined on all $\er^n$. It follows that
\begin{eqnarray*}
\nonumber \left( s_\ep^2+\|Du\|_{L^{\infty}(B_\varrho)}^2\right)^{\frac{p}{2}}&\leq &\frac{c}{(r -\varrho)^{\frac{n}{2}}}\left(\int_{B_r} (s_\ep^2+|Du|^2)^{p}\, dy \right)^{\frac{1}{2}}\\&&\  +c\left( s_\ep^2+\|Du\|_{L^{\infty}(B_r)}^2\right)^{\frac{1}{2}}\|{\bf P}^{V}(\cdot,R)\|_{L^\infty(B_R)}
\end{eqnarray*}
and therefore
\begin{eqnarray}
\nonumber \left( s_\ep^2+\|Du\|_{L^{\infty}(B_\varrho)}^2\right)^{\frac{p}{2}}& \leq & \frac{c( s_\ep^2+\|Du\|_{L^{\infty}(B_r)}^2)^{\frac{p}{4}}}{(r -\varrho)^{\frac{n}{2}}}\left(\int_{B_r} (s_\ep^2+|Du|^2)^{\frac{p}{2}}\, dy \right)^{\frac{1}{2}}\\&&\qquad  +c\left( s_\ep^2+\|Du\|_{L^{\infty}(B_r)}^2\right)^{\frac{1}{2}}\|{\bf P}^{V}(\cdot,R)\|_{L^\infty(B_R)}\,.\label{dopop}
\end{eqnarray}
By using Young's inequality we obtain
\begin{eqnarray*}
\nonumber \left( s_\ep^2+\|Du\|_{L^{\infty}(B_\varrho)}^2\right)^{\frac{p}{2}}&\leq &\frac{1}{2}\left( s_\ep^2+\|Du\|_{L^{\infty}(B_r)}^2\right)^{\frac{p}{2}} \\&&  +\frac{c}{(r -\varrho)^{n}}\int_{B_R} (s_\ep^2+|Du|^2)^{\frac{p}{2}}\, dy+ c\|{\bf P}^{V}(\cdot,R)\|_{L^\infty(B_R)}^{\frac{p}{p-1}}\,.
\end{eqnarray*}
We now apply Lemma \ref{simpfun} with the choice
$$
\varphi(t):=\left( s_\ep^2+\|Du\|_{L^{\infty}(B_t)}^2\right)^{\frac{p}{2}}
$$
and deduce that
\eqn{aaee}
$$
\left( s_\ep^2+\|Du_\ep\|_{L^{\infty}(B_{R/2})}^2\right)^{\frac{p}{2}}\leq c \mean{B_R} (s_\ep^2+|Du_\ep|^2)^{\frac{p}{2}}\, dy  + c \|{\bf P}^{V_\ep}(\cdot,R)\|_{L^\infty(B_R)}^{\frac{p}{p-1}}\,,
$$
with a constant $c$ depending only on $n,p,\ratio$; we recovered here the full notation with the subscript $\ep$ everywhere. This is the final a priori estimate for $Du_\ep$ we were looking for.

{\em Step 5: Passage to the limit and conclusion.} This is now a rather standard procedure; for completeness we
 briefly recall the convergence argument. We consider the weak formulations
$$
\int_{\Omega'} \langle a_\ep(Du_\ep), D\varphi\rangle \, dx = \int_{\Omega'} V_\ep\varphi \, dx
$$
and choose $\varphi =u_\ep-u$ as a testing function. Using \rif{asp}, \rif{mony} and Young's inequality in a standard way we obtain
$$
\int_{\Omega'} |Du_\ep|^p\, dx \leq c \int_{\Omega'} (|Du|^p+s^p)\, dx + c\int_{\Omega'} |V||u_\ep-u|\, dx \,,
$$
for a constant $c \equiv c(n,p,\ratio)$. We now distinguish two cases; the first is when $p\geq 2$, then we use Young's inequality with $\delta \in (0,1)$ and Poincar\'e's inequality in order to have
$$
\int_{\Omega'} |V||u_\ep-u|\, dx \leq  \delta\int_{\Omega'} (|Du_\ep|+|Du|)^{p}\, dx+c(\delta ) \int_{\Omega'} |V|^{p'}\, dx
$$
so that, after properly choosing $\delta$, the last two estimates and $p'\leq 2$ give
$$
\int_{\Omega'} |Du_\ep|^p\, dx \leq c \int_{\Omega'} (|Du|^p+s^p)\, dx +c \int_{\Omega'} (1+|V|^{2})\, dx\,.
$$ In the case $1<p< 2$ by using Young's inequality with $\delta \in (0,1)$ and Sobolev embedding we find
$$
\int_{\Omega'} |V||u_\ep-u|\, dx \leq  \delta\int_{\Omega'} (|Du_\ep|+|Du|)^{p}\, dx+c(\delta ) \left(\int_{\Omega'} |V|^{(p^*)'}\, dx\right)^{\frac{p'}{(p^*)'}}
$$
so that, after properly choosing $\delta$ we have
$$
\int_{\Omega'} |Du_\ep|^p\, dx \leq c \int_{\Omega'} (|Du|^p+s^p)\, dx +c \left(\int_{\Omega'} (1+|V|^{m})\, dx\right)^{\frac{p'}{(p^*)'}}\,.
$$Here we are using the standard notation, i.e.\ by $p^*$ we denote Sobolev's conjugate exponent $np/(n-p)$ if $p< n$, respectively any number larger than $2$
otherwise; accordingly $(p^*)'$ denotes its H\"older
conjugate, that is $p^*/(p^*-1)$.
In any case we come up with bound
\eqn{uunn}
$$\|Du_\ep\|_{L^p(\Omega)}\leq c $$ which is uniform in $\ep$. We are now ready to prove that
\eqn{forte}
$$
u_\ep \to u \qquad \mbox{strongly in} \ W^{1,p}(\Omega')\,.
$$
The fact that both $u_\ep$ and $u$ are solutions now gives
\begin{align}
\nonumber \int_{\Omega'} \langle a_\ep(Du_\ep)&-a_\ep(Du), D\varphi\rangle \, dx \\
&= \int_{\Omega'} \langle a(Du)-a_\ep(Du), D\varphi\rangle \, dx + \int_{\Omega'} (V_\ep-V)\varphi \, dx \label{euno}
\end{align}
which holds whenever $\varphi \in C^\infty_0(\Omega)$; we test \rif{euno} with $\varphi=u_\ep-u$.
We notice that by \rif{unic} and \rif{uunn} it follows that
\eqn{edue}
$$
III_\ep :=\left|\int_{\Omega'} \langle a(Du)-a_\ep(Du), D(u_\ep -u)\rangle \, dx \right|\to 0\,.
$$
As for the left-hand side in \rif{euno}, by \rif{mon3} we have
\eqn{intF}
$$
\int_{\Omega'} |W(Du_\ep)-W(Du)|^2 \, dx\leq c
\int_{\Omega'} \langle a_\ep(Du_\ep)-a_\ep(Du),D(u_\ep -u)\rangle \, dx \,.
$$ Now we observe that
\eqn{conin}
$$
\int_{\Omega'} |Du_\ep-Du|^p \, dx\leq c \left(\int_{\Omega'} \langle a_\ep(Du_\ep)-a_\ep(Du), D(u_\ep -u)\rangle \, dx \right)^{\min\{1,p/2\}}\,,
$$
holds for a constant $c$ independent on $\ep$. Indeed, in the case $p \geq 2$ \rif{conin} is a trivial consequence of \rif{mon2}; in the case $1< p< 2 $ H\"older's inequality and \rif{V} give
\begin{eqnarray*}
&& \int_{\Omega'} |Du_\ep-Du|^p \, dx\\ && \leq  c\left(\int_{\Omega'} (|Du_\ep|^p+|Du|^p) \, dx \right)^{\frac{2-p}{2}}
 \left(\int_{\Omega'}  |W(Du_\ep)-W(Du)|^2 \, dx \right)^{\frac{p}{2}}
\end{eqnarray*}
and again \rif{conin} holds by mean of \rif{mon3}, \rif{uunn} and \rif{intF}. Finally, we turn estimating the last term in \rif{euno}; when $p\geq 2$  we have
$$
\left|\int_{\Omega'} (V_\ep-V)(u_\ep-u) \, dx\right|\leq c\|V_\ep-V\|_{L^{p'}(\Omega')}\| Du_\ep-Du\|_{L^{p}(\Omega')}\,,
$$
Combining this last estimate with \rif{euno}-\rif{conin} and using Young's inequality yields
$$
\| Du_\ep-Du\|_{L^{p}(\Omega')}\leq c \|V_\ep-V\|_{L^{p'}(\Omega')}^{p'} + cIII_\ep\,,
$$
so that \rif{forte} follows as $p'\leq 2$ when $p \geq2$.
When $p<2$ we estimate
$$
\left|\int_{\Omega'} (V_\ep-V)(u_\ep-u) \, dx\right|\leq c\|V_\ep-V\|_{L^{(p^*)'}(\Omega')}\| Du_\ep-Du\|_{L^{p}(\Omega')}\,,
$$
for a constant independent of $\ep$.
Using this last inequality together with \rif{euno}-\rif{edue} we arrive at
$$
\| Du_\ep-Du\|_{L^{p}(\Omega')}^2\leq c \|V_\ep-V\|_{L^{(p^*)'}(\Omega')}\| Du_\ep-Du\|_{L^{p}(\Omega')} + cIII_\ep\,.
$$
Once again \rif{forte} follows via Young's inequality and taking \rif{VVV} into account. Finally, \rif{apl} follows letting $\ep \to 0$ in \rif{aaee} and using lower semicontinuity to deal with the left and side of the inequality as follows:
\begin{eqnarray}
\nonumber\|Du\|_{L^{\infty}(B_{R/2})}^p& \leq &\liminf_{\ep \to 0}\,( s_\ep^2+\|Du_\ep\|_{L^{\infty}(B_{R/2})}^2)^{\frac{p}{2}}\\ \nonumber &\leq& c \lim_{\ep \to 0} \, \mean{B_R} (s_\ep^2+|Du_\ep|^2)^{\frac{p}{2}}\, dy  + c \|{\bf P}^{V}(\cdot,R)\|_{L^\infty(B_R)}^{\frac{p}{p-1}}\\ &=& c\mean{B_R} (s^2+|Du|^2)^{\frac{p}{2}}\, dy  + c \|{\bf P}^{V}(\cdot,R)\|_{L^\infty(B_R)}^{\frac{p}{p-1}}
\,.\label{semi}
\end{eqnarray}
The proof of Theorem \ref{main01} is complete.
\begin{remark}\label{discrete2} In the last line of \rif{discrete} we have used the following estimation, which is standard when dealing with non linear potentials as ${\bf P}^V$:
\begin{eqnarray*}
{\bf P}^{\tilde V}(x, 2R) & = &\sum_{j=0}^{\infty} \int_{R_{j+1}}^{R_j} \left(\frac{|\tilde V|^2(B(x,\varrho))}{\varrho^{n-2}}\right)^{\frac{1}{2}}\, \frac{d\varrho}{\varrho}\\
&\geq & \sum_{j=0}^{\infty} \int_{R_{j+1}}^{R_j} \left(\frac{|\tilde V|^2(B(x,R_{j+1}))}{R_j^{n-2}}\right)^{\frac{1}{2}}\, \frac{d\varrho}{\varrho}\\
&\geq & \frac{\log 2}{2^{\frac{n-2}{2}}}\sum_{j=0}^{\infty} \left(\frac{|\tilde V|^2(B(x,R_{j+1}))}{R_{j+1}^{n-2}}\right)^{\frac{1}{2}}=c(n)\sum_{j=1}^{\infty} \left(\frac{|\tilde V|^2(B(x,R_{j}))}{R_{j}^{n-2}}\right)^{\frac{1}{2}}.
\end{eqnarray*}
\end{remark}
\section{Theorem \ref{mainu}}\label{seclo}
We first recall a few basic definitions relevant in order to deal with Lorentz spaces; we shall use the notion of decreasing rearrangement. Let $\mu\colon \Omega \to \er^k$, with $k \in \en$, being a measurable map such that
$
|\{x \in \Omega \, :\, |\mu(x)|>t \}| < \infty$ for every $t>0$.
Again we assume that $\mu$ is extended to the whole $\er^n$ letting $\mu \equiv 0$ outside $\Omega$. The decreasing rearrangement $\mu^*\colon [0,\infty]\to [0,\infty]$ is pointwise defined by
$$
\mu^*(s):=\sup\, \{t\geq 0 \, : \, |\{x \in \er^n \, : \, |\mu(x)|>t\}|>s\}\,.
$$
This is in other words the (unique) non-increasing, right continuous decreasing function which is equi-distributed with $|\mu(\cdot)|$, i.e.
$
|\{|\mu|> t\}|=|\{\mu^*>t\}|
$
holds whenever $t \geq 0$.
Now, the usual definition of the Lorentz space $L(\gamma, q)$, for $\gamma \in (0,\infty)$ and $q \in (0,\infty)$ prescribes that
$$
[\mu]_{L(\gamma,q)}:=\left(\frac{q}{\gamma}\int_0^\infty \left(\mu^*(\varrho)\varrho^{1/\gamma}\right)^q\, \frac{d\varrho}{\varrho}\right)^{\frac{1}{q}}< \infty\,.
$$
Moreover, by Fubini's theorem it also follows that
\eqn{equiequi}
$$
[\mu]_{L(\gamma,q)}=\left(q\int_0^\infty (\lambda^\gamma|\{|\mu|> \lambda\}|)^\frac{q}{\gamma}\, \frac{d\lambda}{\lambda}\right)^{\frac{1}{q}}\,.
$$
Lorentz spaces refine the standard Lebesgue spaces in the sense that the second index tunes the first in the following sense
whenever $0 < q < t  < r <  \infty$
we have, with continuous embeddings, that
$$
L^r \equiv L(r,r)\subset L(t,q) \subset L(t,t) \subset L(t,r)\subset L(q,q)\equiv L^q\,,$$
while all the previous inclusions are strict. We shall use a characterization of Lorentz spaces using an averaged version of $\mu^*$, due to Hunt \cite{Hunt}; let us consider, for $s>0$, the following maximal operator:
\eqn{muss}
$$
\mu^{**}(s):=\frac{1}{s}\int_0^s \mu^{*}(t)\, dt
$$
and, accordingly, for $q < \infty$
$$
\|\mu\|_{L(\gamma,q)}:=\left(\frac{q}{\gamma}\int_0^\infty \left(\mu^{**}(\varrho)\varrho^{1/\gamma}\right)^q\, \frac{d\varrho}{\varrho}\right)^{\frac{1}{q}}\,.
$$
Then - see for instance \cite[Theorem 3.21]{steinweiss} - it holds that
\eqn{oneil}
$$
[\mu]_{L(\gamma,q)}\leq \|\mu\|_{L(\gamma,q)}\leq c(\gamma,q)[\mu]_{L(\gamma,q)}\qquad \qquad \mbox{for}\ \gamma >1\,.
$$
The following Lemma can be obtained exactly as \cite[Lemma 3.1]{DMc}.
\begin{lemma}\label{potd} Let $V \in L^2(\er^n)$; for every $R >0$ it holds that
\eqn{mpotl}
$$\sup_{x \in \Omega} {\bf P}^V(x,R)\leq c \int_0^{2\omega_nR^n} \left(\left(|V|^{2}\right)^{**}(\varrho )\varrho^{\frac{2}{n}}\right)^{\frac{1}{2}}\,  \frac{d\varrho}{\varrho} \,,$$
where the constant $c$ depends only on $n$ and $\omega_n:= |B_1|$ is the measure of the unit ball in $\er^n$.
\end{lemma}
\begin{proof}[Proof of Theorem \ref{mainu}] We again recall that, up to letting $V\equiv 0$ outside $\Omega$, we can assume that $V \in L(n,1)(\er^n)$. We first observe that
$$
\int_0^{\infty} \left(\left(|V|^{2}\right)^{**}(\varrho )\varrho^{\frac{2}{n}}\right)^{\frac{1}{2}}\,  \frac{d\varrho}{\varrho} \leq \||V|^2\|_{L(n/2,1/2)}^{1/2}\,,
$$
and this in view of \rif{oneil} and $n>2$. The last inequality together with \rif{mpotl} give
$$
\sup_{x \in \Omega} {\bf P}^V(x,R)\leq c\||V|^2\|_{L(n/2,1/2)}^{1/2}\,.
$$
In turn we observe that from the definition of Lorentz spaces we have
$
V\in L\left(n,1\right)$ iff $|V|^{2}  \in L(n/2,1/2)
$
and moreover, by mean of \rif{equiequi}, it is easy to see that
$$
[|V|^2]_{L(n/2, 1/2)}=[V]_{L(n, 1)}^2\,.
$$
Connecting the last two inequalities and using \rif{oneil} again we finally obtain
\eqn{localo}
$$
\sup_{x \in \Omega} {\bf P}^V(x,R)\leq c\|V\|_{L(n,1)}\,.
$$
At this point the local boundedness of $Du$ follows by Theorem \ref{main01} and the last inequality; moreover, estimate \rif{stimalo} follows from \rif{apl} and a standard covering argument using a localization of \rif{localo}.
\end{proof}
\section {Theorem \ref{main02}}\label{thmain02}
Here we give the modifications to the proof of Theorem \ref{main01} which are necessary to deal with the vectorial case $N >1$, therefore obtaining Theorem \ref{main02}; the point is that the structure assumption \rif{uh} allows to recover the Caccioppoli estimate \rif{cacci} also in the vectorial case. Needless to say, the approximation scheme remains essentially unchanged and we are going to show that inequality \rif{cacci} holds with the same meaning of $v\equiv v_\ep$ introduced in \rif{subso}; the only difference is that in order to consider approximating regularized problems we shall use the vector fields in \rif{appvec} rather than those in \rif{reg}, in order to keep for the regularized vector fields the crucial structure condition \rif{uh}. Then the rest of the proof proceeds unchanged with respect to the scalar case $N=1$. For this
reason the proof is structured following the various steps made in
the proof of Theorem \ref{subso}.

{\em Step 1: Approximation.} The approximation is similar to the one settled down for the proof of Theorem \ref{main01}; the truncation of the potential is the one in \rif{truncation} but it must be of course done componentwise, that is
\eqn{truncation2}
$$
V_{\ep}^\alpha(x):=\max\{\min\{V^\alpha(x),1/\ep\}, -1/\ep\}\,,\qquad \alpha \in \{1,\ldots,N\}\,,
$$
while the approximation of the vector field will be done using the regularized vector fields defined in \rif{appvec}.
According to the convention already used for the proof of Theorem \ref{main01} we shall abbreviate $g(\cdot)\equiv g_\ep(\cdot)$, $a(\cdot)\equiv a_\ep(\cdot)$ and $V^i \equiv V_\ep^i$, eventually recovering the full notation when convenient.

{\em Step 2: Caccioppoli inequality for $v$.} We keep the notation introduced for the proof of Theorem \ref{main01}. We restart from \rif{sys1}, that this time takes the form
$$
\int_{\Omega}  a_i^\alpha(Du)D_i\varphi^\alpha\, dx   = \int_{\Omega} \varphi^\alpha V^\alpha \, d x\,,
$$
for $\alpha \in \{1,\ldots,N\}$. We again use $D_s \varphi$ instead of $\varphi$ for $s \in \{1,\ldots,n\}$; integration by parts then yields
\eqn{l1v}
 $$
\int_{\Omega} \partial_{z_j^\beta}a_i^\alpha(Du)D_jD_su^\beta D_i\varphi^\alpha \, dx = -\int_{\Omega} D_s\varphi^\alpha V^\alpha \, dx \,.
$$
for $\alpha,\beta \in \{1,\ldots,N\}$. As a test function we again take $
\varphi:= \eta^2(v-k)_+ D_su
$ and sum up over $s \in \{1,\ldots,n\}$; we have
\begin{eqnarray*}
 I_1 + I_2 + I_3 &:=&  \int_{\Omega} \eta^2\partial_{z_j^\beta}a_i^\alpha(Du) D_jD_su^\beta D_iD_su^\alpha (v-k)_+ \, dx \\
 && \qquad   + \int_{\Omega} \eta^2\partial_{z_j^\beta}a_i^\alpha(Du) D_jD_su^\beta D_su^\alpha D_i(v-k)_+ \, dx\\
 && \qquad   + 2\int_{\Omega} \eta\partial_{z_j^\beta}a_i^\alpha(Du) D_jD_su^\beta  D_su^\alpha (v-k)_+D_i\eta \, dx\\
    &=&   - \int_{\Omega}  \eta^2D_sD_su^\alpha (v-k)_+ V^\alpha\, dx - \int_{\Omega}  \eta^2D_su^\alpha D_s(v-k)_+V^\alpha \, dx \\
   && \qquad   - 2\int_{\Omega} \eta D_s\eta D_su^\alpha (v-k)_+V^\alpha\, dx =: II_1+ II_2+ II_3\,.
\end{eqnarray*}We will just have to estimate the terms on the left hand side; the estimation for the ones appearing on the right hand side
will be then analogous to the one shown in the previous section; see estimates \rif{IIuno}-\rif{pdopo}. As in \rif{i5} by ellipticity we have
\eqn{i52}
$$
I_5:= \int_{\Omega} \eta^2H^{p-2}(Du)|D^2u|^2 (v-k)_+ \, dx\leq c I_1\,.
$$
On the other hand let us observe that \rif{uh} implies
\eqn{identity}
$$
\partial_{z_j^\beta}a_i^\alpha(z) = 2
g'(|z|^2) z_i^{\alpha}z_j^{\beta} +
g(|z|^2) \delta_{ij}\delta^{\alpha\beta}\,,
$$
where the $\delta_{ij}\delta^{\alpha\beta}$ denotes the Kronecker's symbol. Moreover, let us define
\eqn{identity2}
$$
\tilde{a}_{i,j}^{\alpha,\beta}(x):= \frac{\partial_{z_j^\beta}a_i^\alpha(Du(x))}{(s_\ep^2 + |Du(x)|^2)^{\frac{p-2}{2}}}\,,
$$
which is again a uniformly elliptic matrix, and uniformly with respect to
 $\ep$. Summing upon all repeated indexes and using Einstein's summation yields
\begin{eqnarray*}
 I_2 &=&\int_{\Omega} \eta^2\partial_{z_j^\beta}a_i^\alpha(Du) D_jD_su^\beta D_su^\alpha D_i(v-k)_+ \, dx\\
  & = & \int_{\Omega} \eta^2 2g'(|Du|^2) D_iu^{\alpha}D_ju^{\beta}
D_{j}D_su^{\beta}D_su^{\alpha} D_i(v-k)_+  dx\\
&&\quad  +  \int_{\Omega}\eta^2
g(|Du|^2)\delta_{ij}\delta^{\alpha\beta} D_{j}D_su^{\beta}D_su^{\alpha}D_i(v-k)_+\,
   dx =: I_{2,1}+I_{2,2}\,.
  \end{eqnarray*}
Now, using that
$
D_s v =p H^{p-2}(Du) D_jD_su^\beta D_ju^\beta,
$
we have
\begin{eqnarray*}
I_{2,1}&= &\sum_{\alpha,\beta, i,j,s}\int_{\Omega} \eta^22g'(|Du|^2) D_iu^{\alpha}D_su^{\alpha}
D_{j}D_su^{\beta}D_ju^{\beta} D_i(v-k)_+\, dx\\
&= &{\textstyle \frac{1}{p}}\sum_{\alpha, i,s}\int_{\Omega} \eta^22 H^{2-p}(Du)g'(|Du|^2) D_iu^{\alpha}D_su^{\alpha}
D_s(v-k)_+ D_i(v-k)_+\, dx\\
&= &{\textstyle \frac{1}{p}}\sum_{\alpha, i,j}\int_{\Omega} \eta^22H^{2-p}(Du)g'(|Du|^2) D_iu^{\alpha}D_ju^{\alpha}
D_i(v-k)_+ D_j(v-k)_+\, dx
  \end{eqnarray*}and
\begin{eqnarray*}
I_{2,2}&= &\sum_{\alpha, i,j,s}\int_{\Omega} \eta^2g(|Du|^2) \delta_{ij}\delta^{\alpha\alpha}D_{j}D_su^{\alpha}D_ju^{\alpha} D_i(v-k)_+\, dx\\
&= &{\textstyle \frac{1}{p}}\int_{\Omega} \eta^2H^{2-p}(Du)g(|Du|^2) \delta_{ij}D_i(v-k)_+D_j(v-k)_+\, dx\,.
  \end{eqnarray*}
Therefore, connecting the last three equalities and using \rif{identity}-\rif{identity2} we have
\begin{eqnarray*}
I_2 & =&{\textstyle \frac{1}{p}}\int_{\Omega} \eta^2H^{2-p}(Du)\big[2g'(|Du|^2) D_iu^{\alpha}D_ju^{\alpha}
\\ && \hspace{3cm}+g(|Du|^2) \delta_{ij}\delta^{\alpha\alpha}\big] D_i(v-k)_+D_j(v-k)_+\, dx\\
   & = &  {\textstyle \frac{1}{p}} \int_{\Omega} \eta^2 \tilde a_{i,j}^{\alpha,\alpha} D_i(v-k)_+D_j(v-k)_+
 \, dx\,.
\end{eqnarray*}
As a consequence we gain
$$
\nu \int_{\Omega} |D (v-k)_+|^2\eta^2 \, dx \leq pI_2\,.
$$
In a similar way we have
\begin{eqnarray*}
 I_3 &=&2\int_{\Omega} \eta(v-k)_+\partial_{z_j^\beta}a_i^\alpha(Du) D_jD_su^\beta D_su^\alpha D_i\eta \, dx\\
  & =  & 2\int_{\Omega} \eta(v-k)_+2g'(|Du|^2) D_iu^{\alpha}D_ju^{\beta}
D_{j}D_su^{\beta}D_su^{\alpha} D_i\eta \, dx\\
&&\hspace{2cm}+  2\int_{\Omega}\eta
g(|Du|^2)\delta_{ij}\delta^{\alpha\beta} D_{j}D_su^{\beta}D_su^{\alpha}D_i\eta \,
   dx \\
  &  =&  {\textstyle \frac{2}{p}}\int_{\Omega} \eta(v-k)_+H^{2-p}(Du)\big[2g'(|Du|^2) D_iu^{\alpha}D_ju^{\alpha}\\ && \hspace{3cm}
+g(|Du|^2) \delta_{is}\delta^{\alpha\alpha}\big]D_j(v-k)_+ D_i\eta\, dx
 \\
  & =   &  {\textstyle \frac{2}{p}}\int_{\Omega} \eta(v-k)_+ \tilde a_{i,j}^{\alpha,\alpha}(Du) D_j(v-k)_+ D_i\eta
 \, dx\,.
\end{eqnarray*}
Therefore using Young's inequality with $\delta \in (0,1)$ we have
\begin{eqnarray*}
 |I_3| &\leq &c \int_{\Omega} \eta|D\eta|(v-k)_+|D(v-k)_+| \, dx \\ &
\leq &\delta  \int_{\Omega} \eta^2 |D (v-k)_+|^2\, dx +  c(\delta)\int_{\Omega} |D \eta|^2(v-k)_+^2 \, dx\,.
\end{eqnarray*}
Connecting the previous estimates and choosing $\delta$ small enough in order to reabsorb terms we once again obtain
$$I_5 +  \int_{\Omega} \eta^2|D (v-k)_+|^2 \, dx\leq c\int_{\Omega} |D \eta|^2(v-k)_+^2 \, dx + c|II_1+II_2+III_3|\,.
$$
At this point we may proceed exactly as after \rif{inter} and we finally arrive at \rif{cacci}.

{\em Step 3: Conclusion.} Inequality \rif{cacci} in turn implies the a priori estimate \rif{aaee} exactly as for the proof of Theorem \ref{main01}. The passage to the limit applies as in the scalar case. Finally the Lipschitz regularity assertion under the assumption $V \in L(n,1)$ follows exactly as in Section \ref{seclo}.

\section{A priori estimates for Theorems \ref{maind}-\ref{maind3}}\label{apcomp}
Here we show a priori estimates for solutions in the context of Theorems \ref{maind}-\ref{maind2}; a final remark clarifies the case of estimates for Theorem \ref{maind3}, that is the vectorial case $N>1$. The estimates derived here will be combined with the approximation scheme of the next section and for this reason will be given for a priori more regular solutions; specifically, we shall argue under the regularity assumptions of Step 2 in Theorem \ref{main01}. Therefore we assume to have a solution $u$ of \rif{aapp2}
such that \rif{serve}-\rif{serve2} are satisfied; moreover, we shall assume that $V$ is bounded. Finally, the vector field $a(\cdot)$ will satisfy \rif{asp} with $s>0$.

We shall modify the proof given in Section \ref{secth}, restarting from \rif{l1}, which now becomes
 $$
\int_{\Omega} \partial_{z_j}a_i(Du)D_jD_su D_i\varphi \, dx = -\int_{\Omega} b(x,u,Du)VD_s\varphi  \, dx \,,
$$
whenever $s \in \{1,\ldots,n\}$.
We develop the left hand side as after \rif{l1} and arrive at \rif{bbb} where in the right hand side terms the function $V$ must be replaced by $b(x,u,Du)V$. In turn
in the estimation of the terms $II_1$, $II_2$ and $II_3$ we shall use the bound
$$
\eta|b(x,u,Du)Du|\leq \eta\left(\Gamma + \|Du\|_{L^{\infty}({\rm supp}\, \eta)}\right)^{q+1}\,.
$$
Therefore the new estimates are now
\begin{align*}
 |II_1| &\leq  c\int_{\Omega}\eta^2 |D^2u| (v-k)_+ |b(x,u,Du)||V|\, dx \\
 &\leq  c
 (s+\Gamma+\|Du\|_{L^{\infty}})^{q+\frac12}
 \int_{\Omega}\eta^2
 [(v-k)_+]^\frac{1}{2}H(Du)^{\frac{p-2}{2}} |D^2u||V |\, dx\\
 & \leq \delta I_5+ c(\delta)( s+\Gamma+\|Du\|_{L^{\infty}})^{2(q+1)}\int_{\Omega}\eta^2 |V |^2\, dx\,,
\end{align*}
$$
|II_2| \leq \delta \int_{\Omega} |D(v-k)_+|^2\eta^2\, dx  + c(\delta)
\left(\Gamma+\|Du\|_{L^{\infty}({\rm supp}\, \eta)}\right)^{2(q+1)} \int_{\Omega}\eta^2  |V|^2\, dx\,,$$
and
$$ |II_3| \leq c\int_{\Omega} (v-k)_+^2|D\eta|^2\, dx + c \left(\Gamma+\|Du\|_{L^{\infty}({\rm supp}\, \eta)}\right)^{2(q+1)} \int_{\Omega} \eta^2 |V|^2\, dx\,. $$
Using such estimates and proceeding as after \rif{pdopo} we finally arrive at \rif{cacci}, but this time with
\eqn{VV}
$$\tilde V := \left( s+\Gamma+\|Du\|_{L^{\infty}(B_{R})}\right)^{q+1}V\,.$$
Having arrived at this stage we proceed as for Steps 2 and 3 in the proof of Theorem \ref{main01} and conclude that
$$
|v(x)|\leq c \left(\mean{B(x,R)} |v|^2\, dy \right)^{\frac{1}{2}} + c\,{\bf P}^{\tilde{V}}(x,2R)\,,
$$
holds whenever $B(x,R) \subset \Omega'$;
in turn with \rif{VV} this implies, after some elementary manipulations, the following analog of \rif{cover1}:
\begin{eqnarray}
\nonumber \big(s +\Gamma+|Du(x)|\big)^p
&\leq &c\left(\mean{B(x,\tr)} \big(s+\Gamma+|Du|\big)^{2p}\, dy \right)^{\frac{1}{2}}\\&&\qquad +c\left( s +\Gamma+\|Du\|_{L^{\infty}(B(x,\tr))}\right)^{q+1}
{\bf P}^{V}(x,\tr)\,,\label{cover1g}
\end{eqnarray}
whenever $B(x, \tr) \subset \Omega'$.
We now proceed as after \rif{dopop} and arrive at
\begin{eqnarray}
\nonumber && \big( s +\Gamma+\|Du\|_{L^{\infty}(B_\varrho)}\big)^{p}\\ && \nonumber \qquad
\leq \frac{c\left( s +\Gamma+\|Du\|_{L^{\infty}(B_r)}\right)^{\frac{2p-t}{2}}}{(r -\varrho)^{n/2}}\left(\int_{B_r} \big(s+\Gamma+|Du|\big)^{t}\, dy \right)^{\frac{1}{2}}\\
&&\qquad \qquad  +c\left( s+\Gamma+\|Du\|_{L^{\infty}(B_r)}\right)^{q+1}\|{\bf P}^{V}(\cdot,R)\|_{L^\infty(B_R)}\,,\label{ssss1}
\end{eqnarray}
where $t <2 p$ is a positive number and $c$ is a constant which depends only on $n,p,\ratio$.
We now distinguish two cases; the first is when $q < p-1$. In this case we use Young's inequality obtaining
\begin{eqnarray}
\nonumber && \left( s+\Gamma+\|Du\|_{L^{\infty}(B_\varrho)}\right)^{p}\leq {\textstyle \frac{1}{2}}
\left( s+\Gamma+\|Du\|_{L^{\infty}(B_r)}\right)^{p} \\&&\qquad \qquad  +\frac{c}{(r -\varrho)^{np/t}}\left(\int_{B_R} \big(s+\Gamma+|Du|\big)^{t}\, dy\right)^\frac{p}{t}+ c\|{\bf P}^{V}(\cdot,R)\|_{L^\infty(B_R)}^{\frac{p}{p-q+1}}\,,\label{cover11111}
\end{eqnarray}
where $c$ now also depends on $q$.
We now apply again \rif{simpfun} and conclude with
\begin{eqnarray}
\nonumber \left( s+\Gamma+\|Du_\ep\|_{L^{\infty}(B_{R/2})}\right)^{p} & \leq  & c \left(\mean{B_R} \big(s+\Gamma+|Du_\ep|\big)^{t}\, dy \right)^{\frac{p}{t}}\\ && \qquad  + c \|{\bf P}^{V_\ep}(\cdot,R)\|_{L^\infty(B_R)}^{\frac{p}{p-q+1}}\,.\label{aaeegg}
\end{eqnarray}
Taking $t=p$
the a priori estimate \rif{aes2} follows.
The second case is the critical one, that is when $q = p-1$; in this case instead of \rif{ssss1} we have
\begin{eqnarray}
\nonumber && \left( s+\Gamma+\|Du\|_{L^{\infty}(B_\varrho)}\right)^{p}\\ && \nonumber \qquad \leq \frac{c\left( s+\Gamma+\|Du\|_{L^{\infty}(B_r)}\right)^{\frac{2p-t}{2}}}{(r -\varrho)^{n/2}}\left(\int_{B_r} \big(s+\Gamma+|Du|\big)^{t}\, dy \right)^{\frac{1}{2}}\\&&\qquad \qquad  +c_2\left( s+\Gamma+\|Du\|_{L^{\infty}(B_r)}\right)^{p}\|{\bf P}^{V}(\cdot,R)\|_{L^\infty(B_R)}\,,\label{ssss2}
\end{eqnarray}
and
we recall that $c_2$ only depends on $n,p,\ratio$. At this stage we cannot
use Young's inequality to evaluate the last term in \rif{ssss2}; instead we require the smallness condition
on the potential, i.e. we need to assume that the number $\ep_0$ in \rif{smallii} satisfies
$$
c_2 \ep_0 \leq {\textstyle \frac{1}{4}}\,,
$$
and this allows us to reabsorb the second term appearing on the right hand side on the left.
Proceeding in this way and using Lemma \ref{simpfun} we this time conclude with
\eqn{ap}
$$
s+\Gamma+\|Du\|_{L^{\infty}(B_{R/2})}\leq c \left(\mean{B_R} \big(s+\Gamma+|Du|\big)^{t}\, dy\right)^{\frac{1}{t}} \,.
$$
Again the usual a priori estimate for $Du$ follows taking $t=p$.
\begin{remark} Needless to say the above computations apply to the vectorial cases considered in \rif{uh} and in Theorem \ref{main02}. This simply follows by combining the approach of this section with the estimates derived in Section \ref{thmain02} and settles the a priori estimates needed for Theorem \ref{maind3}.
\end{remark}
\begin{remark} It is clear estimates \rif{aaeegg} and \rif{ap} continue to hold when $t >p$, by simply applying H\"older's inequality to the right hand side integral.
\end{remark}

 \section{Theorems \ref{maind}-\ref{maind3}: approximation and proof}\label{secap}
In the previous sections we have built a priori estimates for a priori regular solutions to \rif{ddd}$_1$; in this section we complete the proof by nesting them in a suitable approximation and convergence scheme. We shall directly give the proof simultaneously for Theorems \ref{maind} and \ref{maind2}, therefore in the following we shall assume that
\eqn{alt2}
$$
\|V\|_{L^\gamma(\Omega')}\leq c_0< \infty \qquad \qquad\gamma = \frac{np}{np-nq-n+p}\leq n\,.
$$
Note that $\gamma < n$ iff $q < p-1$. At a certain stage we shall distinguish the case $q=p-1$ (Theorem \ref{maind}) and the case $q < p-1$ (Theorem \ref{maind2}). More specifically, in the case $q=p-1$ we will need to impose an additional restriction on the constant $c_0$ appearing in \rif{alt2} and this will lead to assume \rif{smallii}; in the case $q<p-1$ no additional restriction will be imposed and therefore this will lead to assume simply $\|V\|_{L^\gamma(\Omega')}< \infty$ as in \rif{smalliianc}.

The proof goes briefly as follows: we build an approximation scheme by solving problems with a truncated and therefore bounded - but not uniformly bounded - right hand sides $\mu_\ep$; this allows to get local regular solutions $u_\ep$. We first prove a uniform bound on the $W^{1,p}$-norm of the the approximating solutions $u_\ep$. Then, using this bound we prove a uniform $L^1$-estimate of the right hand sides $\mu_\ep$, which then up to a subsequence converge to (a vector valued measure) $\mu$ in the sense of measures; at this stage we forget about the specific structure the right hand sides and apply methods form the theory of measure data problems to get a first strong convergence - in $W^{1,1}$ - for the approximating solutions. Then combining this with the local $W^{1,\infty}$ a priori estimates of the previous section we gain the strong convergence of the solutions $u_\ep$ locally in $W^{1,t}$ to $u$, for every $t < \infty$. In turn this allows to identify the limiting right hand side $\mu =b(x,u,Du)V$ and to show that $u$ solves \rif{ddd}, while the local estimates follows passing to the limit those found in Section \ref{apcomp}.

In the rest of the proof we shall denote by $\ep$ a specific sequence of positive numbers $\ep \equiv \ep_k$ with $k \in \en$, such that $\ep_k$; from time to time we shall need to extract a subsequence; this will not be relabeled and we shall keep on denoting it by $\ep$.

{\em Step 1: Approximation scheme.} The approximation scheme now goes as follows: we consider the componentwise truncated potentials $V_{\ep}(\cdot)$  defined as in \rif{truncation2},
and moreover we let
\eqn{truncb}
$$
b_\ep(x,u,z):= \frac{b(x,u,z)}{1+ \ep |b(x,u,z)|}\,.
$$
Note that in anyway we have that
$$
|b_\ep(x,u,z)|\leq 1/\ep \qquad \mbox{and}\qquad
|b_\ep(x,u,z)V_\ep|\leq \sqrt{N}/\ep^{2}\,.
$$
Moreover
\eqn{unic22}
$$
b_\ep \to b \qquad \mbox{uniformly on compact subsets of} \ \Omega \times \er^N \times \ma\,.
$$
Accordingly, we define $u_\ep \in u_0+W^{1,p}_0(\Omega, \er^N)$
as the a solution to the Dirichlet problem
\eqn{appkak}
$$
\left\{
    \begin{array}{ccc}
    -\divo \ a_\ep(Du_\ep)=b_\ep(x,u_\ep,Du_\ep)V_{\ep} &  \mbox{in}& \ \Omega\\
        u_\ep= u_0&\mbox{on} &\ \partial \Omega\,.
\end{array}\right.
$$
The vector field $a_\ep(\cdot)$ has been introduced in Section \ref{smoothing} and is defined according to \rif{reg} in the scalar case $N=1$, and according to \rif{appvec} in the vectorial case $N> 1$, when in fact we assume \rif{uh} is in force for the original vector field $a(\cdot)$.
The existence follows again by standard monotonicity methods - see for instance \cite{Lions} and in particular \cite[Th\'eor\`eme 2]{LL} - and standard regularity theory provides the Lipschitz continuity
of the solution; for this we refer for instance to \cite{DB1, DB2, U}.

{\em Step 2: Uniform coercivity bounds}. Here we establish
uniform bounds - with respect to $\ep$ - for the sequence $\{Du_\ep\}$ in Lebesgue spaces. We distinguish between the cases $p< n$ and $p=n$.

{\em The case $p<n$}.
Let us first show that, under the assumptions of Theorems \ref{maind}-\ref{maind2} that there exists a constant $c$, depending on $n,N,p,\ratio, V, \Gamma,\Omega, \|Du_0\|_{L^p}$ but otherwise independent of $\ep$, such that
\eqn{bbound}
$$
\|Du_\ep\|_{L^p(\Omega)}\leq c \,.
$$
We consider the weak formulation
\eqn{weakap}
$$
\int_{\Omega} \langle a_\ep(Du_\ep), D \varphi\rangle \, dx = \int_{\Omega} b_\ep(x,u_\ep,Du_\ep)V_{\ep} \varphi \, dx
$$
and test it with $\varphi = u_\ep-u_0$;
using the growth conditions and performing elementary manipulations we are lead to
\eqn{match}
$$
 \int_{\Omega} |Du_\ep|^p\, dx  \leq  c\int_{\Omega}(s_\ep+|Du_0|)^p\, dx + c\int_{\Omega} |b(x,u_\ep,Du_\ep)||u_0-u_\ep||V|\, dx \,,
$$
where the constant $c$ depends only on $n,N,p,\ratio$. As usual we denote by
$
p^*=
np/(n-p)$ for $p< n
$
the usual Sobolev conjugate exponent. Notice that with such a notation we always have that $\gamma >(p^*)'$ as long as $q>0$ - that is the case under the present assumptions -
and moreover
$$
\left(\frac{\gamma}{(p^*)'}\right)':=
\frac{\gamma (np-n+p)}{\gamma (np-n+p) - np}= \frac{np-n+p}{nq}\,.
$$
We finally notice that
\eqn{servedopo}
$$
q(p^*)'\left(\frac{\gamma}{(p^*)'}\right)'= p \,.
$$
We now estimate the right hand side of \rif{match} via H\"older's and Young's inequalities, where $\delta \in (0,1)$ has to be chosen later:
\begin{eqnarray*}
&& \int_{\Omega} |b(x,u_\ep,Du_\ep)||u_0-u_\ep||V|\, dx \leq c\int_{\Omega} (\Gamma+|Du_\ep|)^{q}|u_0-u_\ep||V|\, dx\\
&& \qquad \leq c \left(\int_{\Omega} \big[(\Gamma+|Du_\ep|)^{q}|V|\big]^{(p^*)'}\, dx \right)^{\frac{1}{(p^*)'}}
\left(\int_{\Omega} |u_0-u_\ep|^{p^*}\, dx \right)^{\frac{1}{p^*}}\\
&&\qquad \leq \delta \int_{\Omega} \left(|Du_0|^p+|Du_\ep|^{p}\right)\, dx \\ && \hspace{2cm}+c (\delta) \left(\int_{\Omega} \big[(\Gamma+|Du_\ep|)^{q}|V|\big]^{(p^*)'}\, dx
\right)^{\frac{p}{(p^*)'(p-1)}}\,.
\end{eqnarray*}
In turn we notice - again by H\"older's inequality - that
\begin{eqnarray}
&& \left(\int_{\Omega} \big[(\Gamma+|Du_\ep|)^{q}|V|\big]^{(p^*)'}\, dx \right)^{\frac{p}{(p^*)'(p-1)}}\nonumber \\
&&  \leq c \left(\int_{\Omega} |V|^\gamma\, dx \right)^{\frac{p}{\gamma(p-1)}}\left(\int_{\Omega}\big(\Gamma+|Du_\ep|\big)^{q(p^*)'\left(\frac{\gamma}{(p^*)'}\right)'}\, dx\right)^{\frac{p}{(p^*)'(p-1)}\left(1/\left(\frac{\gamma}{(p^*)'}\right)'\right)} \nonumber
\\
&& \  \leq  c\|V\|_{L^\gamma(\Omega)}^{\frac{p}{p-1}}\left(\int_{\Omega}(\Gamma+|Du_\ep|)^{p}\, dx\right)^{\frac{q}{p-1}}\,.\label{holdserve}
\end{eqnarray}
The constant $c$ appearing in the last estimate only depends on $n,p,\ratio, q,\Omega$, while in the last line we used \rif{servedopo} and
$$
\frac{p}{(p^*)'(p-1)}\frac{1}{( \gamma/ (p^*)')'}
=\frac{q}{p-1}\,,
$$
that in turn follows from \rif{servedopo}. Joining the last three estimates above all in all give
\begin{align}
 \nonumber \int_{\Omega}&|b(x,u_\ep,Du_\ep)||u-u_\ep||V|\, dx \\
  & \leq  \delta \int_{\Omega} \big( |Du_0|^p+|Du_\ep|^{p}\big)\, dx   + c\|V\|_{L^\gamma(\Omega)}^{\frac{p}{p-1}}\left(\int_{\Omega}(\Gamma+|Du_\ep|)^{p}\, dx\right)^{\frac{q}{p-1}}\label{arr}
\end{align} where $c$  depends on $n,p,\ratio, q,\Omega$ and $\delta$.
Combining  this with \rif{match} we arrive at
\begin{align}
 \nonumber \int_{\Omega}|Du_\ep|^{p}\, dx
  \leq & c_3\delta \int_{\Omega} |Du_\ep|^{p}\, dx
  +c \int_{\Omega}\big(s_\ep+ |Du_0|\big)^{p}\, dx\\
  &+ c_4\|V\|_{L^\gamma(\Omega)}^{\frac{p}{p-1}}\left(\int_{\Omega}(\Gamma+|Du_\ep|)^{p}\, dx\right)^{\frac{q}{p-1}},\label{adjust1}
\end{align}
where $c_3$ and $c_4$ depend both on $n,p,\ratio, q,\Omega$, while $c_4$ depends additionally also on $\delta$.
We further distinguish two cases. The first is the one occurring in Theorem \ref{maind2}, that is when $q<p-1$; in this case
we use Young's inequality with $\tilde\delta \in (0,1)$ to separate the third term appearing in the right-hand side of the preceding
inequality
to deduce that
$$
    \int_{\Omega} |Du_\ep|^p\, dx \leq (c\delta+\tilde\delta) \int_{\Omega} |Du_\ep|^p\, dx + c\int_{\Omega} (s_\ep+\Gamma +|Du_0|)^p\, dx
    + c\|V\|_{L^n(\Omega)}^{\frac{p}{p-q+1}}  \,,
$$
where the constant $c$ depends on $n,p,\ratio, q,\Omega,\delta,\tilde{\delta}$.
Choosing $\delta,\tilde\delta$ small enough to obtain $c\delta +\tilde\delta\leq 1/2$ we conclude with \rif{bbound}.
When $q=p-1$ we are in the situation of Theorem \ref{maind}
and we go back to
\rif{adjust1}. Note that $\gamma =n$ in this case. Now we first choose choose $\delta$ small enough in order to reabsorb  the term
$c_3\delta\int_\Omega |Du_\ep|^pdx$ in the left hand side. This amounts in choosing $c_3\delta\le 1/2$.  This fixes the constant $c_4$
in dependence of $\delta$ and we obtain from \rif{adjust1} immediately that
$$
    \int_{\Omega}|Du_\ep|^{p}\, dx
    \leq
    2c \int_{\Omega}\big(s_\ep+ |Du_0|\big)^{p}\, dx
  + 2c_4\|V\|_{L^n(\Omega)}^{\frac{p}{p-1}}\int_{\Omega}(\Gamma+|Du_\ep|)^{p}\, dx\, .
$$
Next we choose
$c_0$ from \rif{smallii} small enough in order to have
\eqn{smally2}
$$
2^pc_4\|V\|_{L^n(\Omega)}^{\frac{p}{p-1}}\leq 2^pc_4c_0^{\frac{p}{p-1}}\leq {\textstyle \frac{1}{2}}
$$
and again we conclude with \rif{bbound}. Note that the size of the constant $c_0$ appearing in \rif{smallii} is exactly determined in \rif{smally2} and this justifies the content of Remark \ref{compute}.

{\em Case $p=n$}. In this case we show that there is an exponent $t$ - which will be chosen properly close to $n$ - satisfying
 \eqn{ntp}
 $$
 n < t < p_0
 $$
and depending only $n,p,\ratio,p_0,q, \Omega$, but otherwise independent of $\ep$, such that the following analogue of \rif{bbound}:
\eqn{bboundnn}
$$
\|Du_\ep\|_{L^t(\Omega)}\leq c \,,
$$
holds. The constant $c$ in the last inequality will depend only on the fixed data $n,N,p,\ratio, V,\Gamma, \Omega,\|Du_0\|_{L^{p_0}}$ and $t-n$, being otherwise independent of $\ep$; we recall that the exponent $p_0>n$ has been defined in \rif{datum}. Applying Theorem \ref{gehringt} with the choices $\tilde{a}\equiv a_\ep$, $v \equiv u_\ep$, $v_0 \equiv u_0$, $g \equiv b_\ep(x,u_\ep,Du_\ep)V_\ep$, we conclude with
\eqn{maggiore}
$$
Du_\ep \in L^{p_1}(\Omega, \ma)\,,\qquad \qquad p_1 \equiv p_1\left(n,N,\ratio,[\partial \Omega]_{C^{0,1}}\right)\in(n,p_0)\,.
$$
Now we select $t>n$, which is eventually going to be chosen very close to $n$, and appeal to Theorem \ref{ISs}, that we apply to $w = u_\ep-u$ with the choice $\delta = t-n$; the exact value of the number $t$ will be chosen throughout the proof in such a way to determine the dependence on the various constant stated after \rif{ntp}. We initially take $t$ close enough to $n$ to have
\eqn{initial}
$$
n< t\leq p_1\,.
$$
By Theorem \ref{ISs} we find
\eqn{eccole}
$$\varphi \in W^{1,\frac{t}{t-n+1}}_0(\Omega)\qquad \mbox{and}\qquad H\in L^{\frac{t}{t-n+1}}(\Omega, \er^n)$$ such that
$$
|D(u_\ep-u_0)|^{t-n}D(u_\ep-u_0) = D\varphi + H\,,
$$
and moreover the inequalities
\eqn{Iw1}
$$
\|H\|_{L^{\frac{t}{t-n+1}}(\Omega)}\leq c(n,\Omega)(t-n)\|D(u_\ep-u_0)\|_{L^t(\Omega)}^{t-n+1}$$
and
\eqn{Iw2}
$$\|D\varphi\|_{L^{\frac{t}{t-n+1}}(\Omega)}^{\frac{t}{t-n+1}}\leq c(n,\Omega)\|D(u_\ep-u_0)\|_{L^t(\Omega)}^t
$$
hold. We test \rif{weakap} written in the form
$$
\int_{\Omega} \langle a_\ep(Du_\ep)- a_\ep (Du_0), D \varphi\rangle \, dx = \int_{\Omega} b(x,u_\ep,Du_\ep)V_{\ep} \varphi \, dx -\int_{\Omega} \langle  a_\ep (Du_0), D \varphi\rangle \, dx
$$
with $\varphi$ defined in \rif{eccole}, observing that this choice is admissible by \rif{maggiore} and \rif{initial}; keep also in mind \rif{datum}. We proceed estimating the resulting terms. We use monotonicity inequalities in \rif{mon3}-\rif{mony} - which hold by \rif{aspep} - to estimate
\begin{align}
 \nonumber \int_{\Omega}& |Du_\ep|^t\, dx
   \leq  c\int_{\Omega}(s_\ep+|Du_0|)^t\, dx + c\int_{\Omega} |b(x,u_\ep,Du_\ep)||\varphi||V|\, dx\\
  &  +
  c\int_{\Omega} (s_\ep+|Du_\ep  |)^{n-1}|H|\, dx +c\int_\Omega (s_\ep+|Du_0  |)^{n-1}|D(u_\ep-u_0)|^{t-n+1}\, dx\,.\label{mmatch2}
\end{align}
Observe that to derive the previous estimate we applied Young's inequality with exponents $t/(t-n+1)$ and $t/(n-1)$.
Let us first estimate the third term in the right hand side of \rif{mmatch2}. By H\"older's inequality and \rif{Iw1} we have
\begin{align}
\nonumber \int_{\Omega} (s_\ep+|Du_\ep|)^{n-1}|H|\, dx &\leq \left(\int_{\Omega} (s_\ep+|Du_\ep|)^{t}\, dx \right)^{\frac{n-1}{t}}
\left(\int_{\Omega} |H|^{\frac{t}{t-n+1}}\, dx \right)^{\frac{t-n+1}{t}}\\
&\leq  c(n,\Omega)(t-n)\int_{\Omega} \big(s_\ep+|Du_\ep|+|Du_0|\big)^{t}\, dx\,.\label{H}
\end{align}
By more trivial means, and in particular by Young's inequality with $\delta \in (0,1)$, we have
\begin{align}
\nonumber
    \int_{\Omega} (s_\ep+|Du_0|)^{n-1}|D(u_\ep&-u_0)|^{t-n+1}\, dx \\
    &\leq \delta  \int_{\Omega} |Du_\ep|^{t}\, dx + c(\delta)\int_{\Omega} (s_\ep+|Du_0|)^{t}\, dx \,.\label{H2}
\end{align}
We proceed by estimating the last integral appearing in \rif{mmatch2}; to this aim, also accordingly to \rif{alt2}, we introduce
\eqn{ecco1}
$$
\gamma:= \frac{n}{n-q}\,,
$$
and then we further reduce $t$ in order to have
\eqn{sceltat}
$$
n < t < \frac{n(n-1)}{n-q-1}\,.
$$
Notice that the last choice is possible since we are assuming $q > 0$. Denoting by $r$ denotes conjugate exponent to $t/(n-1)$, i.e.
\eqn{eccor}
$$
    r:= \frac{t}{t-n+1}<n\,,
$$
 - the last inequality being a consequence of the fact that $t>n$ - we note that \rif{sceltat} implies
\eqn{gammar}
$$
\gamma > (r^*)'\,.
$$
Here, as usual we have that $r^*=nr/(n-r)$, a definition which makes sense by \rif{eccor}. Now we have, by mean of H\"older's, Sobolev's and
Young's inequality with $\delta \in (0,1)$, and via \rif{Iw2}
\begin{align}
 \nonumber\int_{\Omega} &|b(x,u_\ep,Du_\ep)||\varphi||V|\, dx \leq c\int_{\Omega} \big(\Gamma+|Du_\ep|\big)^{q}|\varphi||V|\, dx\\
&\nonumber \le c \left(\int_{\Omega} \big[(\Gamma+|Du_\ep|)^{q}|V|\big]^{(r^*)'}\, dx \right)^{\frac{1}{(r^*)'}}
\left(\int_{\Omega} |\varphi|^{r^*}\, dx \right)^{\frac{1}{r^*}}\\
&\nonumber \le c\left(n, \Omega, (t-n)\right) \left(\int_{\Omega} \big[(\Gamma+|Du_\ep|)^{q}|V|\big]^{(r^*)'}\, dx \right)^{\frac{1}{(r^*)'}}
\left(\int_{\Omega} |D\varphi|^{r}\, dx \right)^{\frac{1}{r}}\\
&\nonumber\leq \delta \int_{\Omega} |D\varphi|^{r}\, dx +c(\delta,\Omega,(t-n)) \left(\int_{\Omega} \left[(\Gamma+|Du_\ep|)^{q}|V|\right]^{(r^*)'}\, dx \right)^{\frac{r}{(r^*)'(r-1)}}\\
&\leq \delta \int_{\Omega} (|Du_0|^{t}+|Du_\ep|^t)\, dx +c \left(\int_{\Omega} \left[(\Gamma+|Du_\ep|)^{q}|V|\right]^{(r^*)'}\, dx \right)^{\frac{r}{(r^*)'(r-1)}},\label{big}
\end{align}
for a constant $c=c(\delta,\Omega,(t-n))$; we observe that such a constant blows-up when $t \downarrow n$, being related to the constant occurring in
the Sobolev embedding inequality for the exponent $r$ defined in \rif{eccor}, and in fact $r \uparrow n $ when $t \downarrow n$. In turn, by \rif{gammar} we may further use H\"older's inequality as follows:
\begin{eqnarray}
&& \label{ecco0} \quad \left(\int_{\Omega} \left[(\Gamma+|Du_\ep|)^{q}|V|\right]^{(r^*)'}\, dx \right)^{\frac{r}{(r^*)'(r-1)}}\\
&&  \leq c \left(\int_{\Omega} |V|^\gamma\, dx \right)^{\frac{r}{\gamma(r-1)}}\left(\int_{\Omega}(\Gamma+|Du_\ep|)^{q(r^*)'\left(\frac{\gamma}{(r^*)'}\right)'}\, dx\right)^{\frac{r}{(r^*)'(r-1)}\left(1/\left(\frac{\gamma}{(r^*)'}\right)'\right)}\,.\nonumber
\end{eqnarray}
Keeping \rif{ecco1} in mind we record the identities
\eqn{ecco2}
$$
\left(\frac{\gamma}{(r^*)'}\right)'= \frac{\gamma(nr-n+r)}{\gamma(nr-n+r)-nr}= \frac{nr-n+r}{rq-n+r}=\frac{n(n-1)+t}{qt+(n-t)(n-1)}\,,
$$
\eqn{ecco3}
$$
q(r^*)'\left(\frac{\gamma}{(r^*)'}\right)'= \frac{nqr}{rq-n+r}=\frac{nqt}{qt+(n-t)(n-1)}\,,
$$and
\eqn{ecco4}
$$\frac{r}{(r^*)'(r-1)}\left(1/\left(\frac{\gamma}{(r^*)'}\right)'\right)=\frac{qt+(n-t)(n-1)}{n(n-1)}\,.$$
We start observing that
\eqn{ecco5}
$$
\frac{nqt}{qt+(n-t)(n-1)}\leq t \Longleftrightarrow q\leq n-1
$$
and
\eqn{ecco6}
$$
\frac{qt+(n-t)(n-1)}{n(n-1)}\leq 1 \Longleftrightarrow q \leq n-1
$$
hold, with equality - {\em in both \trif{ecco5} and \trif{ecco6} - occurring iff $q=n-1$}. Moreover, we observe that, thanks to \rif{ecco3} and \rif{ecco5}, inequality \rif{ecco0} gives
\begin{align}
     \nonumber \bigg(\int_{\Omega} \big[(\Gamma&+|Du_\ep|)^{q}|V|\big]^{(r^*)'}\, dx \bigg)^{\frac{r}{(r^*)'(r-1)}}\\
& \leq c\|V\|_{L^\gamma(\Omega)}^{\frac{r}{r-1}}\left(\int_{\Omega}\big(1+\Gamma+|Du_\ep|\big)^{t}\, dx\right)^{\frac{r}{(r^*)'(r-1)}\left(1/\left(\frac{\gamma}{(r^*)'}\right)'\right)}\,,\label{ecco02}
\end{align}
We keep on distinguishing two cases; the first is when $q < n-1$; thanks to \rif{ecco4} and \rif{ecco6} - which holds with the strict inequality when $q < n-1$ - in \rif{ecco02} we may apply Young's inequality with $\xi \in (0,1)$ thereby getting
\begin{align}
\nonumber\bigg(\int_{\Omega} \big[(\Gamma+|Du_\ep|\big)^{q}&|V|\big]^{(r^*)'}\, dx \bigg)^{\frac{r}{(r^*)'(r-1)}}\\
&\leq \xi
\int_{\Omega}\big(1+\Gamma+|Du_\ep|\big)^{t}\, dx + c(\xi) \|V\|_{L^\gamma(\Omega)}^{\frac{n}{n-1-q}}\,.\label{slo1}
\end{align}In the remaining borderline case $q=n-1$ we notice that equality holds in \rif{ecco5}-\rif{ecco6} and therefore, keeping  also \rif{ecco3}-\rif{ecco4} in mind, estimate \rif{ecco0} reduces to
\eqn{slo2}
$$\left(\int_{\Omega} \big[(\Gamma+|Du_\ep|)^{q}|V|\big]^{(r^*)'}\, dx \right)^{\frac{r}{(r^*)'(r-1)}}
\leq c\|V\|_{L^n(\Omega)}^{\frac{r}{r-1}}\int_{\Omega}\big(\Gamma+|Du_\ep|\big)^{t}\, dx\,. $$
We now again distinguish between the cases $q < n-1$ and $q=n-1$.
In the first case, matching \rif{slo1} with \rif{big}, and using the resulting inequality together with \rif{H}-\rif{H2} in \rif{mmatch2} finally yields
\begin{align*}
\int_{\Omega} |Du_\ep|^t\, dx  \leq & c_5\Big[ c(\delta,\Omega, (t-n))\xi + \delta+ (t-n)\Big]\int_{\Omega} |Du_\ep|^t\, dx\\
&+ c_6\int_{\Omega} \big(
s_\ep+\Gamma+1+|Du_0|\big)^t\, dx+ c_7 \|V\|_{L^\gamma(\Omega)}^{\frac{n}{n-1-q}} \,.
\end{align*}
where the constant $c_6$ depends only on $n,p,\ratio, \delta, \Omega$ and $c_5$ only on $n,p,\ratio, \Omega$. Note that the constant $c(\delta,\Omega, (t-n))$ inside the square brackets blows-up when $t \searrow n$ and the same happens to $c_7$, which also blows-up for $\xi \to 0$. Now we first select $\delta$ small enough and $t$ close enough to $n$ to have
\eqn{saty}
$$
c_5\delta + c_5(t-n)\leq {\textstyle \frac{1}{4}}
$$
and then we take $\xi$ small enough in order to have
$$
c_5 \xi c(\delta,\Omega, (t-n)) \leq {\textstyle\frac{1}{4}}\,.
$$
In view of the dependence of the constant $c_4$ on the parameters specified above this finally determines $\delta, \xi$ and especially $r$ as a function of $n,p,\ratio$. Recalling that $t < p_0$ , we conclude with
\eqn{sss}
$$
\int_{\Omega} |Du_\ep|^t\, dx   \leq  {\textstyle \frac{1}{2}}\int_{\Omega} |Du_\ep|^t\, dx+ c\int_{\Omega} (s_\ep+\Gamma+1+|Du_0|)^{p_0}\, dx+ c \|V\|_{L^\gamma(\Omega)}^{\frac{n}{n-1-q}} \,,
$$
and \rif{bboundnn} follows in the case $q< n-1$. For the case $q=n-1$ we use \rif{slo2} with \rif{big}, we match the resulting inequality with \rif{H}-\rif{H2} in \rif{mmatch2}, we get
\begin{align*}
\int_{\Omega} |Du_\ep|^t\, dx  \leq & c_5\left[c(\delta,\Omega, (t-n))\|V\|_{L^n(\Omega)}^{\frac{n}{n-1}} + \delta+ (t-n)\right]\int_{\Omega} |Du_\ep|^t\, dx\\
&\qquad  + c\int_{\Omega} \big(s_\ep+\Gamma+1+|Du_0|\big)^t\, dx \,,
\end{align*}
where again $c_5$ depends only on $n,p,\ratio, \Omega$. We first choose $\delta$  and $t$ in order to satisfy \rif{saty}; to conclude we use the assumption \rif{smallii} - and this is the moment the size of the constant $c_0$ is determined - in order to estimate
$$
c_5c(\delta,\Omega, (t-n))\|V\|_{L^n(\Omega)}^{\frac{n}{n-1}}\leq c_5c(\delta,\Omega, (t-n))c_0^{\frac{n}{n-1}}\leq {\textstyle \frac{1}{4}}\,.
$$
Such choices in \rif{sss} imply
$$\int_{\Omega} |Du_\ep|^t\, dx   \leq  {\textstyle \frac{1}{2}}\int_{\Omega} |Du_\ep|^t\, dx+ c\int_{\Omega} (s_\ep+\Gamma+1+|Du_0|)^{p_0}\, dx\,,
$$
and \rif{bboundnn} follows in the case $q= n-1$ too.

{\em Step 3: Convergence and conclusion of the proof}. We here adapt a few compactness arguments which have been developed in the context of measure data problems, and this will allow for a rapid conclusion. We shall follow the strategy introduced in \cite[Section 4]{DHM}, showing in some detail the modifications needed for its adaptation to our context. A difference with \cite{DHM} is that we are dealing with a sequence of vector fields $a_\ep$, instead that with a fixed one - that is $a_\ep(z)=g_\ep(|z|^2)z$ - as in \cite{DHM}. We also observe that the restriction $2-1/n< p$ appearing in \cite{DHM} does not affect the approximation argument in that paper, only coming from the a priori estimates thereby developed for measure data problems. Here we use the convergence arguments in the full range $p>1$. Let us first prove the convergence argument in the more delicate vectorial case $N>1$, when the additional structure assumption \rif{uh} is in force, therefore completing the proof of Theorem \ref{maind3}. At the end we shall add the remarks necessary to treat the general scalar case $N=1$, when on the other hand \rif{uh} is not assumed.

We now switch to the convergence proof. Let us first prove the convergence argument in the more delicate vectorial case $N>1$, when the additional structure assumption \rif{uh} is in force, and therefore completing the proof of Theorem \ref{maind3}. At the end we shall add the remarks necessary to treat the general scalar case $N=1$, when on the other hand \rif{uh} is not assumed. Summarizing \rif{bbound} and \rif{bboundnn} we can assert that, whenever $q \leq p-1$ and $p \leq n$, the following bound holds:
\eqn{bboundnn2}
$$
\|Du_\ep\|_{L^p(\Omega)}\leq c \,.
$$ The constant $c$ - in general depending on $n,p,\ratio$ - will also depend on $t$ in the case $p=n$, being increasing when $t \to n$. 
By \rif{bboundnn2}, and choosing $t$ suitably close to we have that the measures defined by
\eqn{measuree}
$$
\mu_\ep :=|Du_{\ep}|^{q}V_\ep\, dx
$$
have uniformly bounded masses. This is essentially a consequence of proved in the previous steps; indeed when $p< n$, exactly as in \rif{holdserve} we estimate
\begin{eqnarray}
    \nonumber |\mu_\ep|(\Omega)
    &\leq&  \int_\Omega (\Gamma +|Du_\ep|)^q|V|\, dx\leq c \int_{\Omega}(\Gamma + |Du_\ep|)^\frac{nq}{n-1} + |V|^n\, dx \\ \nonumber
    &\le & c \int_{\Omega}(1+\Gamma + |Du_\ep|)^\frac{n(p-1)}{n-1} + |V|^n\, dx\\ \label{follows} &\le& c\int_{\Omega}(1+\Gamma + |Du_\ep|)^p + |V|^n\, dx\,,
\end{eqnarray}
and the uniform bound on $|\mu_\ep|(\Omega)$ follows from \rif{bboundnn2}; the last inequality in \rif{follows} obviously follows from $p\leq n$. Using again \rif{bboundnn2}, and up to passing to a non-relabeled subsequence, we may assume that
\eqn{varconv}
$$
\left\{
\begin{array}{l}
    u_\ep\rightharpoonup u\quad \mbox{weakly in $W^{1,t}(\Omega ,\er^N)$}\\[5pt]
    u_\ep\to u\quad \mbox{strongly in $L^{p-1}(\Omega ,\er^N)$, and a.e. on $\Omega$,}\\[5pt]
    a_\ep(Du_\ep)\rightharpoonup \overline a\quad\mbox{weakly in $L^\frac{t}{p-1}(\Omega,\er^{Nn} )$,}\\[5 pt]
    | Du_\ep-Du|\rightharpoonup h\quad \mbox{weakly in $L^{t}(\Omega)$},\\[5 pt]
    \mu_\ep\rightharpoonup \mu \quad\mbox{as Radon measures.}
\end{array}\right.
$$
Now we enter the proof given in \cite{DHM}, keeping the notation adopted there as much as possible; a first difference with the notation in \cite{DHM}, is that we use the subscript $\ep$ in place of $k$. Also we shall give the proof in the vectorial case $N > 1$ under the additional assumption \rif{uh}; later on we shall show that such an assumption is not necessary in the scalar case $N=1$.

As in \cite{DHM} we first start considering the case $p\geq 2$. The terms $T_\ep$ in \cite{DHM} (actually denoted by $T_k$ in \cite{DHM}) and defined by
\eqn{monost0}
$$
T_\ep := c\int_{\Omega} |Du_\ep-Du|^p \eta (u_\ep-v)\phi\, dx\,,
$$
can be now estimated as
\eqn{monost}
$$T_\ep \leq I_\ep+II_\ep+III_\ep$$
where
\begin{align*}
    I_\ep&:= \int_\Omega \mu_\ep\psi(u_\ep-v)\phi\ dx -\int_\Omega \langle a_\ep(Du_\ep),\psi (u_\ep-v)\otimes D\phi\rangle\, dx\, ,\\
    II_\ep&:= -\int_\Omega \langle a_\ep(Du_\ep),D(\psi (u_\ep-v))\rangle (1-\eta (u_\ep-v))\phi \, dx\, ,\\
    III_\ep&:= -\int_\Omega \langle a_\ep(Dv),D(u_\ep-v)\rangle \eta (u_\ep-v)\phi\, dx\,.
\end{align*}
Here, exactly as in \cite{DHM}, the map $\psi$ is defined as a \ap vectorial truncation", that is
\eqn{truncv}
$$
\psi(z):= \alpha (|\xi|)\frac{\xi}{|\xi|}\qquad \xi \in \er^N
$$
where $\alpha\colon (0,\infty)\to (0,\infty)$ is a bounded differentiable function such that both $\alpha$ and $\alpha'$ are bounded; moreover, as in \cite{DHM}, $v$ is a Lipschitz map to be chosen. Observe that estimate \rif{monost} follows since the vector fields $a_\ep(\cdot)$ satisfy the monotonicity estimate \rif{mon2} uniformly in $\ep$; this is in turn a consequence of \rif{aspep}. Now, we may proceed as in \cite{DHM} in letting $\ep \to 0$. Due to \rif{appvec} we notice that (see again \cite{DHM} for the definition of $\psi$) whenever $z \in \ma$ and $\xi \in \er^N$ it holds that
\begin{eqnarray}\label{vect}
   \nonumber && \langle a_\ep(z), D\psi(\xi)z\rangle =g_\ep(|z|^2)\langle z, D\psi (\xi)z\rangle \\ & &\qquad =g_\ep(|z|^2)\alpha'(|z|)\left \langle z,\frac{\xi}{|\xi|}\right\rangle^2 + g_\ep(|z|^2)\frac{\alpha (|\xi|)}{|\xi|} \left[ |z| ^2 -
    \left \langle z,\frac{\xi}{|\xi|}\right\rangle^2\right]\geq 0
\end{eqnarray}
so that - applying the previous relation with $z \equiv Du_\ep$ and $\xi \equiv u_\ep$ - $II_\ep$ can be estimated as follows:
\eqn{onlyp}
$$
    II_\ep\le \int_\Omega \langle a_\ep(Du_\ep),D\psi (u_\ep-v)Dv\rangle (1-\eta (u_\ep-v))\phi \, dx\, .
$$
By \rif{varconv} and the fact that $\eta ,\psi$ and $D\psi$ are bounded we can pass in this term as in
\cite[Section 4, p. 361]{DHM} to the limit $\ep\to 0$ obtaining
$$
    \limsup_{\ep\to 0} II_\ep\le \int_\Omega \langle \overline a,(D\psi) (u-v)Dv\rangle (1-\eta (u-v))\phi \, dx\, .
$$
As for the term $III_\ep$ we recall that the map $v$ is always considered to be Lipschitz continuous, and therefore, making also use of \rif{unic} yelds
$$
 \limsup_{\ep\to 0} III_\ep =-\int_\Omega \langle a(Dv),D(u-v)\rangle \eta (u-v)\phi\, dx\,.
$$
The passage to the limit $\ep \to 0$ for $I_\ep$ can be done easily exactly as in \cite{DHM}. The rest of the proof can be done exactly as after \cite[(11)]{DHM}. The final conclusion in our setting is that $h$ defined in \rif{varconv} equals zero and therefore
\eqn{convt}
$$
Du_\ep \to Du \qquad \mbox{strongly in }\ \ L^1(\Omega, \ma)\,.
$$
The argument for the case $1 < p <2 $ is quite similar; according to the arguments in \cite{DHM} the term in \rif{monost0} must be replaced by
$$
T_\ep := c\int_{\Omega} (|Du_\ep|+|Dv|)^{p-2}|Du_\ep-Du|^2 \eta (u_\ep-v)\phi\, dx\,.
$$
On the other hand we observe that from the a priori estimates \rif{aaeegg}-\rif{ap} and the a priori bound in \rif{bboundnn2}, we easily gain that for every open subset $\Omega' \Subset \Omega$ there exists a constant $c$ depending only on $n,N,p,\ratio$ and dist$(\Omega', \partial \Omega)$, but otherwise independent of $\ep>0$, such that
$
\|Du_\ep\|_{L^\infty(\Omega')}< c
$
holds. In turn this estimate together with \rif{convt} and a standard diagonal argument implies
\eqn{convt0}
$$
Du_\ep \to Du \qquad \mbox{strongly in }\ \ L^t_{\loc}(\Omega, \ma)\ \ \mbox{for every} \ t < \infty\,.
$$
The last result together with a standard consequence of Lebesgue's dominated convergence theorem and with \rif{unic} and \rif{unic22}, allows finally to let $\ep \to 0$ in \rif{weakap} - recall that there $\varphi$ has compact support in $\Omega$ - thereby getting
$$
\int_{\Omega} \langle a(Du), D \varphi\rangle \, dx = \int_{\Omega} b(x,u,Du)V \varphi \, dx\,,
$$
that is the limiting map $u$ weakly solves \rif{ddd}.
The proof of the a priori estimates \rif{aes1} and \rif{aes2} follows letting $\ep \to 0$ in \rif{ap} and \rif{aaeegg} respectively - there we take $t=p$ - making also use of \rif{convt0}, as already done for \rif{semi}. The proof in the vectorial case $N>1$ is therefore complete; this means we have Theorem \ref{maind3}.

As for the scalar case $N=1$, that is when \rif{uh} is not in force, we notice that the only point in the convergence argument for vectorial case above where \rif{uh} was used was the estimate in \rif{onlyp}, which in turn is a consequence of \rif{vect}. At this stage we add a couple of remarks. The first is that we can always assume that $a_\ep(0)=0$, this by replacing $a_\ep(z)$ by $a_\ep(z)-a_\ep(0)$ which changes nothing in the problem i.e. $\divo\, (a_\ep(Du_\ep)- a_\ep(0))= \divo \, a_\ep(Du_\ep)$. The advantage is that now we may assume that
\eqn{veramon}
$$\langle a_\ep(z), z\rangle\geq \nu_0|z|^p \,.$$
The second observation is that since now the solutions are scalar valued we can replace the function in \rif{truncv} by
$
\psi(t):= \min \{t,1\}$ for $ t \geq 0$. In this way in \rif{vect} we still have $\langle a_\ep(z),\psi'(\xi)z\rangle =\psi'(\xi)\langle a_\ep(z), z\rangle\geq 0$, this being a consequence of \rif{veramon}. At this point \rif{onlyp} follows and the rest of the proof for Theorems \ref{maind} and \ref{maind2} remains unchanged as for the scalar case.
\section{Possible extensions}
In this section we want to briefly outline a few possible extensions of the above results to operators with more general growth conditions, for instance considered in the paper of Lieberman \cite{Li}. We shall confine ourselves to equations and systems of the type \rif{ppp}, with assumptions \rif{asp} replaced by
\eqn{aspge}
$$
\left\{
    \begin{array}{c}
    |a(z)|\leq Lh(|z|)\,,\quad |\partial a(z)| \leq Lh'(|z|) \\ [3 pt]
    \nu^{-1}h'(|z|)|\lambda|^{2} \leq \langle \partial a(z)\lambda, \lambda
    \rangle
    \end{array}
    \right.
$$
where $h\colon \to [0,\infty)$ is a non-decreasing $C^1((0,\infty))$ such that
\eqn{ellh}
$$
0< \delta_0 \leq \frac{h'(t)t}{h(t)}\leq \Lambda\,.
$$
Note that in the case $h(t)=t^{p-1}$ we recover the standard $p$-Laplacean operator with $\delta_0 =p-1>0$ as long as $p>1$.

The point is now that the boundedness criteria of Theorems \rif{main01}-\rif{main02} hold true and in particular estimate (\ref{apl}) holds in the form
\begin{align}\label{apestimate}
    \||Du|h(|Du|)&\|_{L^\infty(B_{R/2})}\nonumber\\
    &\leq  c \left(\mean{B_R} |Du|h(|Du|)\, dx \right)^{\frac{1}{2}}
    + c\|{\bf P}^V(\cdot, R)\|_{L^\infty(B_{R})}^{\frac{1}{2}}+ c\,.
\end{align}
The proof rest as usual on the approximation argument - that can be realized combining the arguments here with those in \cite{Li} - and with the a priori estimates. We shall here sketch the proof of this last one, which in turn relies on the fact that the Caccioppoli's inequality \rif{cacci} holds with
\eqn{cacciv}
$$
v:= h(|Du|)|Du|\,.
$$
To this aim, going back to the proof of \rif{cacci}, let us observe this fact in the context here
leads to the following analog of \rif{bbb}:
\begin{eqnarray}
 \nonumber && I_1 + I_2 + I_3 :=  \int_{\Omega} \eta^2\tilde{a}_{i,j}H(Du)D_jD_su D_iD_su (v-k)_+ \, dx \\
  \nonumber && \qquad   + \int_{\Omega} \eta^2\tilde{a}_{i,j}H(Du) D_jD_su D_su D_i(v-k)_+ \, dx\\
&& \qquad   + 2\int_{\Omega} \eta\tilde{a}_{i,j}H(Du) D_jD_su  D_su (v-k)_+D_i\eta \, dx = II_1+ II_2+ II_3\,.\label{bbbh}
\end{eqnarray}
This time we have defined
\eqn{HH}
$$
H(|Du|):= h'(|Du(x)|)+h(|Du(x)|)/|Du(x)|
$$ and, obviously
$$
\tilde{a}_{i,j}(x):= \frac{\partial_{z_j}a_i(Du(x))}{H(Du(x))}\,.
$$
Observe that this last matrix is uniformly elliptic - with eigenvalues uniformly bounded from above and below by a constant depending on $\ratio, \delta_0, \Lambda$ - by \rif{aspge} and \rif{ellh}. By observing that this time it is $
D_j v =H(Du) D_jD_suD_su
$ and replacing in the estimates \rif{bbb}-\rif{pdopo} the function $H(Du)^{p-2}$ used there by $H(Du)$ introduced in \rif{HH}, we can proceed estimating as done there as far as the left hand side terms are concerned. As the the right hand side ones, this time we estimate
$$
(v-k)_+ = \sqrt{(v-k)_+}\sqrt{\frac{h(|Du|)}{|Du|}}|Du|\leq \frac{1}{\sqrt{\delta_0}}\|Du\|_{L^{\infty}}\sqrt{(v-k)_+}\sqrt{H(Du)}\,,
$$
so that \rif{IIuno} can be replaced by
$$
 |II_1| \leq \delta I_5+ c\|Du\|_{L^{\infty}}^2\int_{\Omega}\eta^2 |V |^2\, dx\,.
$$
The estimates for the terms $II_2, II_3$ follow exactly as in \rif{IIdue}-\rif{pdopo}. This leads to \rif{cacci} with $\tilde V$ defined as $\tilde V := \|Du_\ep\|_{L^{\infty}(B_{R})}V.$ Proceeding as in the proof of Theorem \ref{main01} we arrive at
\begin{align}
\nonumber h(|Du(x)|)&|Du(x)|\\
&\leq c\left(\mean{B(x,\tr)} (h(|Du|)|Du|)^2\, dy \right)^{\frac{1}{2}} +c \|Du\|_{L^{\infty}(B(x,\tr))}{\bf P}^{V}(x,\tr)\label{cover1h}.
\end{align}
In turn, proceeding as after \rif{cover1} and using the fact that $t \leq h(t)t+1$ we conclude with \rif{apestimate}.

Similar adjustments can be done in the vectorial case $N> 1 $.
\section{Appendix: Basic facts about regularity}\label{app}
Let us now recall a few basic facts about regularity of solutions to equations and systems of the type
$$
\divo\, a(Dw)=f \qquad \mbox{in}\ \Omega'
$$
where we are assuming assumptions \rif{asp} with $s>0$ on the left hand side vector field $a(\cdot)$; on the right hand side we shall assume that
$
 |f|\leq M.$
 These are the equations and systems of the type considered in \rif{Dirapp} when proving the a priori estimates of Theorems \ref{main01}-\ref{maind2}. In order to make sense to all the involved quantities there we needed to have \rif{serve}, that is
\eqn{serveoo}
$$
w \in W^{2,2}_{\loc}(\Omega', \er^N)\cap C^{1, \alpha}_{\loc}(\Omega', \er^N)\,.
$$
We shall now recall how to prove \rif{serveoo}; both in the scalar and in the vectorial case with the additional structure assumption we have that $w \in  C^{1, \alpha}_{\loc}(\Omega', \er^N)$ for some $\alpha>0$ depending on $M$. As far as the higher differentiability is concerned we have that
$$
\int_{\Omega''} (s+|Dw|^2)^{\frac{p-2}{2}}|D^2w|^2\, dx < \infty
$$
whenever $\Omega'' \Subset \Omega'$. When $p\geq 2$ this immediately implies $w \in W^{2,2}_{\loc}(\Omega', \er^N)$ since $s>0$; for the case $1<p<2$ we using the fact that $Dw \in L^\infty(\Omega, \ma)$ we again conclude that $w \in W^{2,2}_{\loc}(\Omega', \er^N)$ and \rif{serveoo} is established. Similar results holds for solutions under the assumptions \rif{aspge}. Good references on such aspects, including the assumptions in \rif{aspge} - are for instance \cite{DB1, DB2, Li}.

{\bf Acknowledgments.} This research has been supported by the ERC grant 207573 \ap Vectorial Problems". The results in this paper were obtained in August 2009 while the second named author was visiting the University of Erlangen-Nuremberg and presented during the conference \ap Regularity for non-linear PDE", held in September 2009 at Scuola Normale Superiore, Pisa.

\end{document}